\documentclass[a4paper,11pt, reqno]{amsart}
\usepackage{cite}

\usepackage{xcolor}
\usepackage{amsmath,amsthm,amscd,amssymb}
\usepackage{latexsym}
\usepackage[colorlinks,citecolor=red,pagebackref,hypertexnames=false]{hyperref}

\theoremstyle{plain}
\newtheorem{theorem}{Theorem}[section]
\newtheorem{lem}[theorem]{Lemma}
\newtheorem{corollary}[theorem]{Corollary}
\newtheorem{proposition}[theorem]{Proposition}

\theoremstyle{definition}
\newtheorem{definition}[theorem]{Definition}

\theoremstyle{remark}
\newtheorem{rem}[theorem]{Remark}

\newtheorem{case[theorem]}{Case}

\def \R{{\mathbb R}}

\def \C{{\mathbb C}}

\newcommand{\dint}{\displaystyle\int}

\def\norm#1.#2.{\lVert#1\rVert_{#2}}

\def\R{\mathbb R}

\def \H{{\mathcal H}}

\title[Sampling and Density Theorems for the Fractional Opdam--Cherednik Transform]{Sampling and Density Theorems for the Fractional Opdam--Cherednik Transform}

\author{Riya Ghosh}
\address{Department of Electrical Engineering, Indian Institute of Technology Bombay, Mumbai 400076, India}
\email{riya74012@gmail.com, 20004493@iitb.ac.in}

\author{Anirudha Poria}
\address{Department of Applied Mathematics, School of Mathematics and Physics, Xi’an Jiaotong-Liverpool University, Suzhou 215123, China}
\email{Anirudha.Poria@xjtlu.edu.cn}

\keywords{Fractional Opdam–Cherednik transform; sampling theorem; Riesz basis; Paley–Wiener space; Landau density conditions; concentration operator}

\subjclass[2020]{Primary 42A85, 42B35; Secondary 94A20, 43A32.}

\date{\today}

\begin{document}
\begin{abstract}
In this paper, we establish sampling and density results for the fractional Opdam--Cherednik transform. Using the Sturm--Liouville structure of Jacobi functions, we construct a Riesz basis associated with the Opdam--Cherednik kernel and derive an explicit sampling formula. The sampling nodes are determined by the zeros of a shifted Jacobi function. We further develop a concentration-operator approach and obtain Landau-type necessary density conditions for sampling and interpolation in fractional Opdam–Cherednik bandlimited spaces. The non-fractional Opdam--Cherednik sampling formula is recovered when $\theta=\pi/2$.
\end{abstract}
\maketitle
\section{Introduction}

Sampling theory is a fundamental topic in signal processing and harmonic analysis concerned with the reconstruction of a function from its values on a discrete set. It plays an important role in telecommunications, signal processing, data conversion, medical imaging, and digital audio and video processing; see, for example, \cite{sha48,sha49,opp97,pro95,lyo11}. The classical Whittaker--Shannon--Kotel'nikov theorem states that a bandlimited function can be reconstructed exactly from uniformly spaced samples; see \cite{zay93}. This fundamental principle has motivated the development of sampling theorems associated with several generalized Fourier transforms. In particular, sampling results have been obtained in the Dunkl setting \cite{var07,osa08} and for the Hankel transform \cite{abr05,hig72}. We refer to \cite{gar00} for a general discussion of orthogonal sampling formulas.

Fractional Fourier analysis provides another important extension of the classical Fourier framework. The fractional Fourier transform allows a continuous transition between different representations of a signal and has been studied extensively both from the mathematical and applied points of view. Early developments can be found in \cite{wie29,nam80,mcb87}, while applications and further developments include quantum mechanics \cite{nam80}, harmonic analysis \cite{zay98}, radar systems \cite{yet03}, digital communications \cite{mar01}, time-frequency analysis \cite{pei07}, and optics and signal processing \cite{oza01}. This has also led to fractional versions of several generalized Fourier transforms, including the fractional Dunkl transform \cite{gha14}, fractional Hankel transform \cite{ker91}, fractional Jacobi transform \cite{sah20}, fractional Stockwell transform \cite{wei21}, fractional Jacobi--Dunkl transform \cite{hao22}, and fractional Opdam--Cherednik transform \cite{bah24}. 

The Opdam--Cherednik transform associated with the Jacobi--Cherednik differential--difference operator is the Fourier transform in the trigonometric Dunkl setting (see \cite{opd95, opd00}). The harmonic analysis related to the Opdam--Cherednik transform differs substantially from the Euclidean Fourier setting: the role of the exponential kernel is played by the Opdam hypergeometric function, which is closely related to Jacobi functions. As the harmonic analysis associated with the fractional Fourier transform and Opdam--Cherednik transform has known remarkable development, it is natural to study the fractional Opdam--Cherednik (FrOC) transform and investigate its applications. The FrOC transform introduces a fractional parameter $\theta\in(0,\pi)$, and its kernel can be written as
$$ \mathcal{G}_{\alpha,\beta}^{\theta}(x,\lambda) = e^{-\frac{i}{2}(x^2+\lambda^2)\cot\theta} G_{\lambda\csc\theta}^{\alpha,\beta}(x), $$
where $G_{\lambda}^{\alpha,\beta}$ is the Opdam hypergeometric function. The FrOC transform is defined by \cite{bah24}
$$ \mathcal H_{\alpha,\beta}^{\theta}f(\lambda) = \int_{\mathbb R} \mathcal G_{\alpha,\beta}^{\theta}(-x,\lambda) f(x)A_{\alpha,\beta}(x)\,dx. $$
For $\theta=\pi/2$, the FrOC transform reduces to the Opdam--Cherednik transform. The inversion and Plancherel properties of this transform provide a natural setting for studying Paley--Wiener spaces and discrete reconstructions.

The main purpose of the present paper is to develop a sampling theory for the FrOC transform and investigate the necessary density of sampling and interpolation sets. This problem is not a direct consequence of the classical sampling theorem. In the present setting, the kernel is not an exponential function and the natural sampling points do not arise from a uniform lattice. Instead, they are determined by the spectral properties of Jacobi functions. Moreover, the fractional parameter introduces both a chirp factor and a rescaling of the spectral variable. Consequently, the construction of an appropriate basis and the identification of the sampling nodes require a separate spectral analysis. For recent results on the sampling theorem in the fractional setting, we refer to \cite{SamplingFrJD}. 

Our first objective is therefore to construct a suitable basis on the interval $(-1,1)$. The Jacobi differential equation allows us to formulate the relevant problem in Sturm--Liouville form. If
$$ \Psi_{\alpha,\beta}(\mu) = \varphi_{\mu}^{\alpha+1,\beta+1}(1), $$
and $0<\mu_1<\mu_2<\cdots$ are its positive zeros, then these parameters are related to the Neumann boundary condition
$$ \big(\varphi_{\mu_j}^{\alpha,\beta}\big)'(1)=0. $$
Using the even and odd parts of the corresponding Sturm--Liouville problem, we construct an orthonormal basis involving Jacobi functions and their derivatives. From this basis, we obtain a Riesz basis generated by the Opdam hypergeometric functions. This structure is then used to derive an explicit sampling formula for the fractional Paley--Wiener space. The sampling points are
$$ s_j=\mu_j\sin\theta, $$
so that the fractional parameter directly changes the location of the sampling nodes. The sampling expansion converges in the Paley--Wiener norm and uniformly on compact subsets. When $\theta=\pi/2$, it reduces to the corresponding sampling formula for the Opdam--Cherednik transform.

A second objective is to determine necessary density conditions for general sampling and interpolation sets. An explicit sampling expansion gives one particular reconstruction set, whereas a density theorem describes how dense any stable sampling set must be and how sparse an interpolation set can be. Landau's density method \cite{Landau67} has also been extended to non-Euclidean transform settings; in particular, Abreu and Bandeira \cite{AbreuLandau12} established necessary density conditions for sampling and interpolation associated with the Hankel
transform. Following that, we introduce the appropriate concentration operators associated with the FrOC transform and study their trace and spectral behaviour. For the spatially bandlimited space, this leads to Landau-type lower and upper density conditions for sampling and interpolation sets, respectively. For $S=[-\Omega,\Omega]$, the critical density is 
$$ \frac{2^{1-2\rho}}{\pi}\Omega\csc\theta. $$
We also consider the spectrally bandlimited Paley--Wiener space arising in the sampling theorem and obtain the corresponding density conclusions under the stated trace-defect condition.

The main contributions of the paper can therefore be summarized as follows:

\begin{enumerate}
\item We introduce the FrOC Paley--Wiener space relevant to the sampling problem and identify the spectral parameters that determine the sampling nodes.

\item Using Sturm--Liouville theory, we construct an orthonormal basis from Jacobi functions and their derivatives and deduce a Riesz basis formed by Opdam hypergeometric functions.

\item We establish an explicit sampling theorem for the FrOC transform, with sampling nodes determined by the zeros of
$\varphi_{\mu}^{\alpha+1,\beta+1}(1).$
\item We develop a concentration-operator approach and establish Landau-type necessary density conditions for sampling and interpolation in FrOC bandlimited spaces.
\end{enumerate}

The paper is organized as follows. In Section \ref{sec2}, we recall the basic properties of Jacobi functions and the Opdam--Cherednik transform that are needed in the sequel, together with the corresponding fractional transform. In Section \ref{sec4}, we develop the Sturm--Liouville basis construction and establish the sampling theorem. Section \ref{Sec5} is devoted to Landau's necessary density conditions for sampling and interpolation, including both the spatially and spectrally bandlimited settings. Finally, we conclude with a brief discussion of the results and possible further directions.

\section{Preliminaries}\label{sec2}

In this section, we give a brief overview of the Jacobi--Cherednik operator and related harmonic analysis. The main references for this section are \cite{ank12,GT,mej14, opd95, opd00, sch08}. However, we will use the same notation as in \cite{por21}.

We recall only the basic facts concerning the Jacobi--Cherednik operator and the Opdam--Cherednik transform that will be used later. We assume throughout that
$$ \alpha\geq\beta\geq-\frac12,\quad \alpha>-\frac12, \quad \rho=\alpha+\beta+1. $$
The Jacobi--Cherednik differential--difference operator is defined by
\[T_{\alpha, \beta} f(x)=\frac{\mathrm{d}}{\mathrm{d} x} f(x)+ \frac{A'_{\alpha, \beta} (x)}{A_{\alpha, \beta} (x)} \left(  \frac{f(x)-f(-x)}{2} \right) - \rho f(-x), \]
where
\begin{align}\label{A_alpha_beta0}
 A_{\alpha, \beta} (x)= (\sinh |x| )^{2 \alpha+1} (\cosh x )^{2 \beta+1}.
\end{align}

For $\lambda\in\mathbb C$, the Opdam hypergeometric function
$G_{\lambda}^{\alpha,\beta}$ is the unique analytic solution of
$$T_{\alpha,\beta}u(x)=i\lambda u(x), \quad u(0)=1. $$
It is related to the Jacobi function by
$$ G_{\lambda}^{\alpha,\beta}(x) = \varphi_{\lambda}^{\alpha,\beta}(x) - \frac{1}{\rho-i\lambda} \frac{d}{dx}\varphi_{\lambda}^{\alpha,\beta}(x), $$
where
$\varphi^{\alpha, \beta}_\lambda (x)={}_2F_1 \left(\frac{\rho+i \lambda}{2}, \frac{\rho-i \lambda}{2} ; \alpha+1; -\sinh^2 x \right) $ is the hypergeometric function. Equivalently,
$$ G_{\lambda}^{\alpha,\beta}(x) = \varphi_{\lambda}^{\alpha,\beta}(x) + \frac{\rho+i\lambda}{4(\alpha+1)} \sinh(2x)\, \varphi_{\lambda}^{\alpha+1,\beta+1}(x).$$
For real $x,\lambda$,
$$\overline{G_{\lambda}^{\alpha,\beta}(x)} = G_{-\lambda}^{\alpha,\beta}(x), $$
and we shall use the estimate
$$|G_{\lambda}^{\alpha,\beta}(x)|\leq 2, \quad x,\lambda\in\mathbb R.$$
For further properties of $G_{\lambda}^{\alpha,\beta}$, we refer to
\cite{ank12,GT,sch08}.
\begin{definition}
Let $\alpha \geq \beta \geq -\frac{1}{2}$ with $\alpha > -\frac{1}{2}$. The Opdam--Cherednik transform $\mathcal{H}_{\alpha, \beta} (f)$ of a function $f \in C_c(\R)$ is defined by
\begin{equation}\label{Helmi1444b} \H_{\alpha, \beta} (f) (\lambda)=\dint_{\R} f(x)\; G^{\alpha, \beta}_\lambda(-x)\; A_{\alpha, \beta} (x) dx \quad \text{for all } \lambda \in \C. \end{equation}
The inverse Opdam--Cherednik transform for a suitable function $g$ on $\R$ is given by
\begin{equation}\label{HHH} \H_{\alpha, \beta}^{-1} (g) (x)= \dint_{\R} g(\lambda)\; G^{\alpha, \beta}_\lambda(x)\; d\sigma_{\alpha, \beta}(\lambda) \quad \text{for all } x \in \R, \end{equation}
where $$d\sigma_{\alpha, \beta}(\lambda)= \left(1- \dfrac{\rho}{i \lambda} \right) \dfrac{d \lambda}{8 \pi |C_{\alpha, \beta}(\lambda)|^2}$$ and
$$C_{\alpha, \beta}(\lambda)= \dfrac{2^{\rho - i \lambda} \Gamma(\alpha+1) \Gamma(i \lambda)}{\Gamma \left(\frac{\rho + i \lambda}{2}\right)\; \Gamma\left(\frac{\alpha - \beta+1+i \lambda}{2}\right)}, \quad \lambda \in \C \setminus i \mathbb{N}.$$
\end{definition}

The OC transform $\mathcal H_{\alpha,\beta}$ extends to an isometric isomorphism
$$ \mathcal H_{\alpha,\beta}: L^2(\mathbb R,A_{\alpha,\beta}) \longrightarrow L^2(\mathbb R,\sigma_{\alpha,\beta}).$$
The Plancherel formula is given by
\begin{equation}\label{eq03}
\dint_{\R} |f(x)|^2 A_{\alpha, \beta}(x) dx=\dint_\R \H_{\alpha, \beta} (f)(\lambda) \overline{\H_{\alpha, \beta} ( \check{f})(-\lambda)} \; d \sigma_{\alpha, \beta} (\lambda),
\end{equation}
where $\check{f}(x):=f(-x)$. In particular, for even $f$, Plancherel formula \eqref{eq03} gives
$$\dint_{\R} |f(x)|^2 A_{\alpha, \beta}(x) dx=\dint_\R |\H_{\alpha, \beta} (f)(\lambda)|^{2} \; \dfrac{d \lambda}{8 \pi |C_{\alpha, \beta}(\lambda)|^2}.$$

\subsection{Fractional Opdam--Cherednik transform}\label{sec3}
Let $0<\theta<\pi$. The fractional Jacobi--Cherednik operator is defined by
$$ T_{\alpha,\beta}^{\theta}f(x) = T_{\alpha,\beta}f(x) +i\cot\theta\,x f(x). $$
It is related to the Jacobi--Cherednik operator by
$$ e^{\frac{i}{2}x^2\cot\theta} \circ T_{\alpha,\beta}^{\theta} \circ e^{-\frac{i}{2}x^2\cot\theta} = T_{\alpha,\beta}. $$
The corresponding fractional Opdam hypergeometric function is
\begin{align}\label{Fr_eq1}
 \mathcal G_{\alpha,\beta}^{\theta}(x,\lambda) = e^{-\frac{i}{2}(x^2+\lambda^2)\cot\theta} G_{\lambda\csc\theta}^{\alpha,\beta}(x),
 \end{align}
which is the unique analytic solution of
$$ T_{\alpha,\beta}^{\theta}f = i\lambda\csc\theta\,f, \quad f(0) = e^{-\frac{i}{2}\lambda^2\cot\theta}. $$
For $\theta=\pi/2$, we recover
$$ \mathcal G_{\alpha,\beta}^{\pi/2}(x,\lambda) = G_{\lambda}^{\alpha,\beta}(x). $$
The FrOC transform of
$f\in L^1(\mathbb R,A_{\alpha,\beta})$ is defined by
$$ \mathcal H_{\alpha,\beta}^{\theta}f(\lambda) = \int_{\mathbb R} \mathcal G_{\alpha,\beta}^{\theta}(-x,\lambda) f(x)A_{\alpha,\beta}(x)\,dx.$$
By \eqref{Fr_eq1},
$$ \mathcal H_{\alpha,\beta}^{\theta}f(\lambda) = e^{-\frac{i}{2}\lambda^2\cot\theta} \mathcal H_{\alpha,\beta} \left( e^{-\frac{i}{2}x^2\cot\theta}f(x) \right) (\lambda\csc\theta).$$
Thus the FrOC transform is obtained from the ordinary Opdam--Cherednik transform by chirp multiplication and a rescaling of the spectral variable. This relation will be used repeatedly in the sequel.

The inversion formula is
$$ f(x) = \int_{\mathbb R} \mathcal H_{\alpha,\beta}^{\theta}f(\lambda) \mathcal G_{\alpha,\beta}^{\theta}(x,\lambda) \,d\sigma_{\alpha,\beta}(\lambda\csc\theta),$$
under the usual integrability assumptions. Moreover,
$\mathcal H_{\alpha,\beta}^{\theta}$ extends to an $L^2$-isomorphism between the spatial and spectral spaces associated with the measures
$$ A_{\alpha,\beta}(x)\,dx \quad\text{and}\quad d\sigma_{\alpha,\beta}(\lambda\csc\theta), $$
respectively. We denote its inverse by
$$ \mathcal T_{\alpha,\beta}^{\theta} = \left(\mathcal H_{\alpha,\beta}^{\theta}\right)^{-1}. $$
Since $\mathcal H_{\alpha,\beta}^{\theta}$ is an isometric isomorphism between the corresponding $L^2$-spaces, it is unitary; hence
$$ \mathcal T_{\alpha,\beta}^{\theta} = \left(\mathcal H_{\alpha,\beta}^{\theta}\right)^{*}. $$
For further properties of the FrOC transform, we refer to \cite{bah24}.

\subsection{Sampling and interpolation sets}
Let $\mathcal X$ be a Hilbert space of functions on $\mathbb R$ for which point evaluations are well defined. A discrete set $\Lambda=\{\lambda_n\}\subset\mathbb R$ is called a sampling set for $\mathcal X$ if there exist constants $A_\Lambda,B_\Lambda>0$ such that
$$ A_\Lambda\|f\|_{\mathcal X}^{2} \le \sum_{n}|f(\lambda_n)|^{2} \le B_\Lambda\|f\|_{\mathcal X}^{2}, \quad f\in\mathcal X. $$
It is called an interpolation set for $\mathcal X$ if, for every sequence $\{a_n\}\in\ell^2$, there exists $f\in\mathcal X$ such that
$$f(\lambda_n)=a_n,\quad n\in \mathbb{Z}.$$

\section{Sampling theorems for the FrOC transform}\label{sec4}
Let $0<\theta<\pi$. We write
$$\mathcal{T}_{\alpha,\beta}^\theta=\left(\mathcal{H}_{\alpha,\beta}^\theta\right)^{-1}.$$
We consider the \textbf{Spectrally Bandlimited Space} or \textbf{Paley--Wiener space} as
\begin{align}\label{defn_PW}
{PW}^{\theta}_{\alpha,\beta}
=\{f\in L^2(\mathbb{R},\sigma_{\alpha,\beta}(\lambda\, \csc{\theta})): \text{ supp}({\mathcal{T}_{\alpha,\beta}^\theta f})\subseteq [-1,1]\}.
\end{align}
For later use, set
\begin{align}\label{Sec4_Eqn1}
 \Psi_{\alpha,\beta}(\mu) = \varphi_{\mu}^{\alpha+1,\beta+1}(1). 
 \end{align}
Since $\varphi_\mu^{\alpha+1,\beta+1}$ is even in $\mu$, the function
$\Psi_{\alpha,\beta}$ is even. Let
$$ 0<\mu_1<\mu_2<\cdots $$
denote its positive real zeros, and put
$$ \mu_{-j}=-\mu_j,\quad s_j=\mu_j\sin\theta, \quad j\in\mathbb Z\setminus\{0\}, $$
together with $s_0=0$. Notice that, by the identity
$$ G_\mu^{\alpha,\beta}(x) = \varphi_\mu^{\alpha,\beta}(x) - \frac{1}{\rho-i\mu} \frac{d}{dx}\varphi_\mu^{\alpha,\beta}(x), $$
we have
\begin{align}\label{Sec4_Eqn2}
\frac{d}{dx}\varphi_\mu^{\alpha,\beta}(x) = - \frac{\mu^2+\rho^2}{4(\alpha+1)} \sinh(2x) \varphi_\mu^{\alpha+1,\beta+1}(x).
\end{align}
Consequently,
\begin{align}\label{Sec4_Eqn3}
\Psi_{\alpha,\beta}(\mu_j)=0 \text{ implies that }\frac{d}{dx} \varphi_{\mu_j}^{\alpha,\beta}(1)=0, \quad j\ne0. 
\end{align}
Thus the nonzero sampling parameters are exactly the real spectral parameters associated with the Neumann condition at $x=1$.

Consider the Sturm–Liouville problem
$$ -\frac{1}{A_{\alpha+1,\beta+1}(x)} \frac{d}{dx} \left( A_{\alpha+1,\beta+1}(x)\frac{du}{dx} \right) =\Lambda u,\quad 0<x<1, $$
with the natural boundary condition at $x=0$ and $u(1)=0$. Its eigenvalues are real, simple, and form an increasing sequence tending to infinity. Since the regular solution is
$$ u(x)=\varphi_\mu^{\alpha+1,\beta+1}(x), \quad \Lambda=\mu^2+(\rho+2)^2, $$
its positive spectral parameters are precisely the positive zeros
$ 0<\mu_1<\mu_2<\cdots $
of $\Psi_{\alpha,\beta}(\mu)=\varphi_\mu^{\alpha+1,\beta+1}(1)$. These zeros are simple.

We first establish the basis structure associated with the spectral parameters $\{\mu_j\}$. This will provide the functional-analytic framework needed for the sampling expansion.
\begin{lem}\label{orthogonal} 
Let $\{\mu_j\}_{j\ge1}$ denote the positive spectral parameters determined by 
$$ \Psi_{\alpha,\beta}(\mu_j) = \varphi_{\mu_j}^{\alpha+1,\beta+1}(1)=0,$$
and set
$$ P_j^{\alpha,\beta} := \int_{-1}^{1} \left|\varphi_{\mu_j}^{\alpha,\beta}(x)\right|^2 A_{\alpha,\beta}(x)\,dx , \quad j\ge1, $$
and
$$ M_0:=\int_{-1}^{1}A_{\alpha,\beta}(x)\,dx. $$
Then
\begin{align}\label{Sec4_Eqn4}
\left\{ \frac{1}{\sqrt{M_0}}, \frac{\varphi_{\mu_j}^{\alpha,\beta}}{\sqrt{P_j^{\alpha,\beta}}}, \frac{(\varphi_{\mu_j}^{\alpha,\beta})'} {\sqrt{(\mu_j^2+\rho^2)P_j^{\alpha,\beta}}} :\ j\ge1 \right\}
\end{align}
is an orthonormal basis of $L^2\big((-1,1),A_{\alpha,\beta}\big).$ Consequently,
$$ \left\{ \frac1{\sqrt{M_0}}, \frac{G_{\mu_j}^{\alpha,\beta}} {\sqrt{P_j^{\alpha,\beta}}}, \frac{G_{-\mu_j}^{\alpha,\beta}} {\sqrt{P_j^{\alpha,\beta}}} :\ j\ge1 \right\} $$
is a Riesz basis of
$L^2\big((-1,1),A_{\alpha,\beta}\big).$
\end{lem}

\begin{proof}
Let us denote
$$ A(x)=A_{\alpha,\beta}(x), \quad \varphi_\mu(x)=\varphi_\mu^{\alpha,\beta}(x). $$
The Jacobi differential equation has the Sturm–Liouville form
\begin{align}\label{Sec4_Eqn6}
\left(A(x)\varphi_\mu'(x)\right)' = -(\mu^2+\rho^2)A(x)\varphi_\mu(x).
\end{align}
By the derivative identity already established in \eqref{Sec4_Eqn2},
\begin{align}\label{Sec4_Eqn6_1}
\varphi_\mu'(x) = -\frac{\mu^2+\rho^2}{4(\alpha+1)} \sinh(2x) \varphi_\mu^{\alpha+1,\beta+1}(x).
\end{align}
Hence
$$ \Psi_{\alpha,\beta}(\mu_j)=0 \quad\Longrightarrow\quad \varphi_{\mu_j}'(1)=0.$$
Thus $\varphi_{\mu_j}$ satisfies the Neumann boundary condition at $x=1$. At the singular endpoint $x=0$, the regular Jacobi solution satisfies
$$ \lim_{x\rightarrow0}A(x)\varphi_{\mu_j}'(x)=0. $$
Consequently, the functions $\varphi_{\mu_j}$ are the nonconstant eigenfunctions of the self-adjoint Sturm–Liouville problem
$$ -\frac1{A(x)} \frac{d}{dx} \left(A(x)\frac{du}{dx}\right) =\lambda u, \quad 0<x<1,$$
with the natural boundary condition at $0$ and $ u'(1)=0. $ The constant function $u_0(x)=1$ is the eigenfunction corresponding to the Sturm–Liouville eigenvalue $0$. Notice that this does not mean that it corresponds to the Jacobi spectral parameter $\mu=0$; the parametrization in \eqref{Sec4_Eqn6} is $\lambda=\mu^2+\rho^2$. By the Sturm–Liouville spectral theorem,
$$ \left\{ 1,\varphi_{\mu_j}:j\ge1 \right\} $$
is a complete orthogonal system in the even subspace of
$L^2((-1,1),A)$. Indeed, if $j\neq k$, applying the Lagrange identity to $\varphi_{\mu_j}$ and $\varphi_{\mu_k}$ gives
$$ (\mu_k^2-\mu_j^2) \int_{-1}^{1} \varphi_{\mu_j}(x)\varphi_{\mu_k}(x)A(x)\,dx = \left[ A(x) \big( \varphi_{\mu_j}'\varphi_{\mu_k} - \varphi_{\mu_j}\varphi_{\mu_k}' \big) \right]_{-1}^{1}. $$
The boundary term vanishes, since
$$ \varphi_{\mu_j}'(\pm1) = \varphi_{\mu_k}'(\pm1)=0. $$
Therefore,
\begin{align}\label{Sec4_eqn6_2}
 \int_{-1}^{1} \varphi_{\mu_j}(x)\varphi_{\mu_k}(x)A(x)\,dx=0, \quad j\neq k. 
 \end{align}

Moreover, integrating \eqref{Sec4_Eqn6} over $[-1,1]$ and using
$\varphi_{\mu_j}'(\pm1)=0$, we obtain
$$ (\mu_j^2+\rho^2) \int_{-1}^{1} \varphi_{\mu_j}(x)A(x)\,dx=0. $$
Since $\mu_j^2+\rho^2>0$,
$$ \int_{-1}^{1} \varphi_{\mu_j}(x)A(x)\,dx=0.$$
Thus every $\varphi_{\mu_j}$ is orthogonal to the constant function. Next, integration by parts gives
\begin{align*} \int_{-1}^{1} |\varphi_{\mu_j}'(x)|^2A(x)\,dx &= -\int_{-1}^{1} \varphi_{\mu_j}(x) \left(A(x)\varphi_{\mu_j}'(x)\right)'dx \\ &= (\mu_j^2+\rho^2) \int_{-1}^{1} |\varphi_{\mu_j}(x)|^2A(x)\,dx. 
\end{align*}
Hence
$$ \int_{-1}^{1} |\varphi_{\mu_j}'(x)|^2A(x)\,dx = (\mu_j^2+\rho^2)P_j^{\alpha,\beta}.$$
Also, for $j\neq k$,
$$\int_{-1}^{1} \varphi_{\mu_j}'(x) \varphi_{\mu_k}'(x)A(x)\,dx = (\mu_j^2+\rho^2) \int_{-1}^{1} \varphi_{\mu_j}(x)\varphi_{\mu_k}(x)A(x)\,dx =0. $$

Thus the derivatives form an orthogonal system. Since
$\varphi_{\mu_j}$ is even, $\varphi_{\mu_j}'$ is odd, so every derivative term is automatically orthogonal to every even term.

It remains to prove completeness in the odd subspace.
From \eqref{Sec4_Eqn6_1},
\begin{align}\label{Sec4_Eqn6_3}
\varphi_{\mu_j}'(x) = -\frac{\mu_j^2+\rho^2}{4(\alpha+1)} \sinh(2x) \varphi_{\mu_j}^{\alpha+1,\beta+1}(x).
\end{align}
Furthermore,
\begin{align}\label{Sec4_Eqn6_4}
A_{\alpha+1,\beta+1}(x) = \frac14\sinh^2(2x)A_{\alpha,\beta}(x).
\end{align}
Consider the multiplication operator
$$ Uh(x) = \frac12\sinh(2x)h(x). $$
By \eqref{Sec4_Eqn6_4},
$$ \|Uh\|_{L^2((-1,1), \, A_{\alpha,\beta})}^2 = \|h\|_{L^2((-1,1), \, A_{\alpha+1,\beta+1})}^2. $$
Conversely, if $f$ belongs to the odd subspace, then $h(x)=\frac{2f(x)}{\sinh{(2x)}}$ is even and
$$ \|h\|_{L^2((-1,1), \, A_{\alpha+1,\beta+1})} = \|f\|_{L^2((-1,1), \,A_{\alpha,\beta})}. $$
Hence, $U$ is onto and a unitary map from the even subspace of $ L^2((-1,1), A_{\alpha+1,\beta+1})$ onto the odd subspace of
$ L^2((-1,1),A_{\alpha,\beta}).$ 

Now the functions
$\varphi_{\mu_j}^{\alpha+1,\beta+1}$, $j\ge1, $
are precisely the eigenfunctions of the corresponding Sturm–Liouville problem with the natural condition at $0$ and the Dirichlet condition
$u(1)=0.$ Indeed,
$ u(1)=0$ implies that $\Psi_{\alpha,\beta}(\mu_j)=0.$
Hence, by the Sturm–Liouville spectral theorem,
$$ \left\{ \varphi_{\mu_j}^{\alpha+1,\beta+1}:j\ge1 \right\} $$
is complete in the relevant even subspace of
$L^2((-1,1),A_{\alpha+1,\beta+1})$. Applying the unitary map $U$, and using \eqref{Sec4_Eqn6_3}, shows that
$$ \left\{ \varphi_{\mu_j}':j\ge1 \right\} $$
is complete in the odd subspace of
$L^2((-1,1),A_{\alpha,\beta})$. Combining the even and odd parts proves that
$$ \left\{ \frac1{\sqrt{M_0}}, \frac{\varphi_{\mu_j}} {\sqrt{P_j^{\alpha,\beta}}}, \frac{\varphi_{\mu_j}'} {\sqrt{(\mu_j^2+\rho^2)P_j^{\alpha,\beta}}} :\ j\ge1 \right\} $$
is an orthonormal basis.
For the Riesz-basis assertion, put
$$ e_j = \frac{\varphi_{\mu_j}} {\sqrt{P_j^{\alpha,\beta}}}, \quad o_j = \frac{\varphi_{\mu_j}'} {\sqrt{(\mu_j^2+\rho^2)P_j^{\alpha,\beta}}}. $$
Using
$$ G_{\mu_j}^{\alpha,\beta} = \varphi_{\mu_j} - \frac{\varphi_{\mu_j}'} {\rho-i\mu_j}, $$
we obtain
$$ \frac{G_{\mu_j}^{\alpha,\beta}} {\sqrt{P_j^{\alpha,\beta}}} = e_j-c_j o_j,$$
where
$$ c_j = \frac{\sqrt{\rho^2+\mu_j^2}} {\rho-i\mu_j} = \frac{\rho+i\mu_j} {\sqrt{\rho^2+\mu_j^2}}, \quad |c_j|=1. $$
Similarly,
$$ \frac{G_{-\mu_j}^{\alpha,\beta}} {\sqrt{P_j^{\alpha,\beta}}} = e_j-\overline{c_j}\,o_j.$$
Thus on the two-dimensional space
$\operatorname{span}\{e_j,o_j\}$, the change of basis is represented by
$$ M_j = \begin{pmatrix} 1&-c_j\\ 1&-\overline{c_j} \end{pmatrix}. $$
Its determinant is
$$ \det M_j = c_j-\overline{c_j} = \frac{2i\mu_j} {\sqrt{\rho^2+\mu_j^2}}. $$
Since $\mu_j\ge\mu_1>0,$
we have
$$ \inf_{j\ge1}|\det M_j| = \frac{2\mu_1}{\sqrt{\rho^2+\mu_1^2}} >0. $$

Moreover, the entries of $M_j$ are uniformly bounded because
$|c_j|=1$. Hence both $M_j$ and $M_j^{-1}$ are uniformly bounded in $j$. Therefore, the block-diagonal transformation sending
$\{e_j,o_j\}_{j\ge1}$
to
$$ \left\{ \frac{G_{\mu_j}^{\alpha,\beta}}{\sqrt{P_j^{\alpha,\beta}}}, \frac{G_{-\mu_j}^{\alpha,\beta}}{\sqrt{P_j^{\alpha,\beta}}} \right\}_{j\ge1} $$
is bounded and boundedly invertible. Since the former, together with $M_0^{-1/2}$, is an orthonormal basis, the latter is a Riesz basis.
This proves the lemma.
\end{proof}

To identify the coefficients in the Riesz-basis expansion explicitly, we next compute the relevant pairings between the Opdam hypergeometric functions.
\begin{lem} Let $\mu,\nu\in\mathbb R$, and define \begin{align}\label{Sec4_Eqn9} \mathcal I(\mu,\nu) := \int_{-1}^{1} G_{\mu}^{\alpha,\beta}(x) G_{\nu}^{\alpha,\beta}(-x) A_{\alpha,\beta}(x)\,dx . \end{align} 
Let 
\[ \kappa_{\alpha,\beta} := \frac{A_{\alpha,\beta}(1)\sinh(2)} {2(\alpha+1)}\]
and $\mu_j\neq0$ be a zero of
$\Psi_{\alpha,\beta}(\mu) = \varphi_{\mu}^{\alpha+1,\beta+1}(1). $
Then, for $\nu\neq\mu_j$,
\begin{align}\label{Sec4_Eqn10}
\mathcal I(\mu_j,\nu)=
i\kappa_{\alpha,\beta}
\varphi_{\mu_j}^{\alpha,\beta}(1)
\frac{(\rho+i\nu)\Psi_{\alpha,\beta}(\nu)}
{\mu_j-\nu}.
\end{align}
At $\nu=\mu_j$,
\begin{align}\label{Sec4_Eqn11}
\mathcal I(\mu_j,\mu_j)=
-i\kappa_{\alpha,\beta}
\varphi_{\mu_j}^{\alpha,\beta}(1)
(\rho+i\mu_j)
\Psi_{\alpha,\beta}'(\mu_j).
\end{align}
Moreover, for every $\nu\in\mathbb R$,
\begin{align}\label{Sec4_Eqn12_thm}
\int_{-1}^{1}
G_{\nu}^{\alpha,\beta}(-x)
A_{\alpha,\beta}(x)\,dx=
\kappa_{\alpha,\beta}
\Psi_{\alpha,\beta}(\nu).
\end{align}
\end{lem}

\begin{proof}
For simplicity, throughout the proof we write
$$ A(x)=A_{\alpha,\beta}(x), \quad \varphi_{\mu}(x)=\varphi_{\mu}^{\alpha,\beta}(x), \quad \Psi(\mu)=\Psi_{\alpha,\beta}(\mu). $$
Since $\varphi_{\nu}$ is even and $\varphi_{\nu}'$ is odd, we have
$$ G_{\nu}^{\alpha,\beta}(-x) = \varphi_{\nu}(x) + \frac{\varphi_{\nu}'(x)} {\rho-i\nu}. $$
Therefore, using
$$ G_{\mu_j}^{\alpha,\beta}(x) = \varphi_{\mu_j}(x) - \frac{\varphi_{\mu_j}'(x)} {\rho-i\mu_j}, $$
and observing that the two mixed terms are odd, we obtain
\begin{align}\label{Sec4_Eqn12}
\mathcal I(\mu_j,\nu)=
R(\mu_j,\nu)-\frac{Q(\mu_j,\nu)}
{(\rho-i\mu_j)(\rho-i\nu)},
\end{align}
where
$$ R(\mu_j,\nu) = \int_{-1}^{1} \varphi_{\mu_j}(x)\varphi_{\nu}(x) A(x)\,dx $$
and
$$ Q(\mu_j,\nu) = \int_{-1}^{1} \varphi_{\mu_j}'(x)\varphi_{\nu}'(x) A(x)\,dx. $$
The Jacobi functions satisfy the differential equation
$$ \big(A(x)\varphi_{\mu}'(x)\big)' = -(\mu^2+\rho^2) A(x)\varphi_{\mu}(x). $$
Applying the Lagrange identity to
$\varphi_{\mu_j}$ and $\varphi_{\nu}$, we obtain
\begin{align}\label{Sec4_Eqn13}
(\nu^2-\mu_j^2)
R(\mu_j,\nu)=
2A(1)
\left[
\varphi_{\mu_j}'(1)\varphi_{\nu}(1)-
\varphi_{\mu_j}(1)\varphi_{\nu}'(1)
\right].
\end{align}
Since $\Psi(\mu_j)=0$, relation \eqref{Sec4_Eqn2} implies
$$ \varphi_{\mu_j}'(1)=0. $$
Hence, whenever $\nu^2\neq\mu_j^2$,
\begin{align}\label{Sec4_Eqn14}
R(\mu_j,\nu)=\frac{
2A(1)\varphi_{\mu_j}(1)\varphi_{\nu}'(1)}{\mu_j^2-\nu^2}.
\end{align}

Furthermore, integration by parts and
$\varphi_{\mu_j}'(\pm1)=0$ give
\begin{align}\label{Sec4_Eqn15}
Q(\mu_j,\nu)=(\mu_j^2+\rho^2)
R(\mu_j,\nu).
\end{align}
Substituting \eqref{Sec4_Eqn14} and \eqref{Sec4_Eqn15} into \eqref{Sec4_Eqn12}, we obtain
$$\mathcal I(\mu_j,\nu)= \frac{ 2A(1)\varphi_{\mu_j}(1)\varphi_{\nu}'(1) } {\mu_j^2-\nu^2} \left[ 1- \frac{\mu_j^2+\rho^2} {(\rho-i\mu_j)(\rho-i\nu)} \right].$$
Since
$$ \mu_j^2+\rho^2 = (\rho-i\mu_j)(\rho+i\mu_j), $$
we have
$$ 1- \frac{\mu_j^2+\rho^2} {(\rho-i\mu_j)(\rho-i\nu)} = -\frac{i(\mu_j+\nu)} {\rho-i\nu}. $$
Thus, for $\nu\neq\pm\mu_j$,
$$ \mathcal I(\mu_j,\nu) = -\frac{ 2iA(1)\varphi_{\mu_j}(1)\varphi_{\nu}'(1) } {(\mu_j-\nu)(\rho-i\nu)}. $$
Using \eqref{Sec4_Eqn2},
$$ \varphi_{\nu}'(1) = - \frac{\nu^2+\rho^2} {4(\alpha+1)} \sinh(2)\Psi(\nu), $$
and the factorization
$$ \nu^2+\rho^2 = (\rho-i\nu)(\rho+i\nu), $$
we obtain
$$ \mathcal I(\mu_j,\nu) = i\kappa_{\alpha,\beta} \varphi_{\mu_j}(1) \frac{(\rho+i\nu)\Psi(\nu)} {\mu_j-\nu}. $$
This proves \eqref{Sec4_Eqn10} for
$\nu\neq\pm\mu_j$. It remains to consider $\nu=-\mu_j$. Since
$\Psi$ is even,
$$ \Psi(-\mu_j)=0, $$
so the right-hand side of \eqref{Sec4_Eqn10} vanishes. On the other hand, from \eqref{Sec4_Eqn12} and \eqref{Sec4_Eqn15},
\begin{align*}
\mathcal I(\mu_j,-\mu_j) &= R(\mu_j,-\mu_j) - \frac{ Q(\mu_j,-\mu_j) }{ (\rho-i\mu_j)(\rho+i\mu_j) } \\ 
&= R(\mu_j,-\mu_j) - \frac{ (\mu_j^2+\rho^2)R(\mu_j,-\mu_j) }{ \mu_j^2+\rho^2 } =0.
\end{align*}
Thus \eqref{Sec4_Eqn10} also holds for
$\nu=-\mu_j$.
Next, letting $\nu\to\mu_j$ in \eqref{Sec4_Eqn10}, and using the simplicity of the zero $\mu_j$, we have
$$ \Psi(\nu) = \Psi'(\mu_j)(\nu-\mu_j) + o(\nu-\mu_j). $$
Therefore,
$$ \mathcal I(\mu_j,\mu_j) = -i\kappa_{\alpha,\beta} \varphi_{\mu_j}(1) (\rho+i\mu_j) \Psi'(\mu_j), $$
which proves \eqref{Sec4_Eqn11}. Finally, since $\varphi_{\nu}'$ is odd,
$$ \int_{-1}^{1} G_{\nu}^{\alpha,\beta}(-x)A(x)\,dx = \int_{-1}^{1} \varphi_{\nu}(x)A(x)\,dx. $$
Integrating the Jacobi equation over $[-1,1]$, we obtain
$$ 2A(1)\varphi_{\nu}'(1) = -(\nu^2+\rho^2) \int_{-1}^{1} \varphi_{\nu}(x)A(x)\,dx. $$
Hence
$$ \int_{-1}^{1} \varphi_{\nu}(x)A(x)\,dx = -\frac{ 2A(1)\varphi_{\nu}'(1) }{ \nu^2+\rho^2 }. $$
Using \eqref{Sec4_Eqn2} once again,
$$ \int_{-1}^{1} G_{\nu}^{\alpha,\beta}(-x) A_{\alpha,\beta}(x)\,dx = \frac{ A_{\alpha,\beta}(1)\sinh(2) }{ 2(\alpha+1) } \Psi_{\alpha,\beta}(\nu), $$
which is precisely \eqref{Sec4_Eqn12_thm}. This completes the proof.
\end{proof}

\begin{corollary}\label{cor:biorthogonal} 
For $j\in\mathbb Z\setminus\{0\}$, set $D_j:=\mathcal I(\mu_j,\mu_j)$.
Then $D_j\neq0$, and
\begin{align}\label{Sec4_biorth1}
\mathcal I(\mu_k,\mu_j)=
D_j\delta_{jk},
\quad j,k\in\mathbb Z\setminus\{0\}.
\end{align}
Moreover,
\begin{align}\label{Sec4_biorth2}
\int_{-1}^{1}
G_{\mu_j}^{\alpha,\beta}(-x)
A_{\alpha,\beta}(x) dx
=0,\quad j\in\mathbb Z\setminus\{0\}.
\end{align}
Consequently, if
$$ g(x) = c_0+\sum_{k\in\mathbb Z\setminus\{0\}} c_kG_{\mu_k}^{\alpha,\beta}(x) $$
is the Riesz-basis expansion of
$g\in L^2((-1,1),A_{\alpha,\beta})$, then
\begin{align}\label{Sec4_coeff}
c_j=\frac{1}{D_j}
\int_{-1}^{1}
g(x)G_{\mu_j}^{\alpha,\beta}(-x)
A_{\alpha,\beta}(x)\,dx,
\quad j\neq0.
\end{align}
\end{corollary}

\begin{proof}
Let $j,k\in\mathbb Z\setminus\{0\}$. If $j\neq k$, then $\mu_j\neq\mu_k$. Applying \eqref{Sec4_Eqn10} with $\mu_k$ in place of $\mu_j$ and $\nu=\mu_j$, we obtain
$$ \mathcal I(\mu_k,\mu_j) = i\kappa_{\alpha,\beta} \varphi_{\mu_k}^{\alpha,\beta}(1) \frac{ (\rho+i\mu_j)\Psi_{\alpha,\beta}(\mu_j) }{ \mu_k-\mu_j }. $$
Since $\mu_j$ is a zero of $\Psi_{\alpha,\beta}$, 
$$ \mathcal I(\mu_k,\mu_j)=0, \quad k\neq j. $$
For $k=j$, by definition,
$$ \mathcal I(\mu_j,\mu_j)=D_j. $$
Hence \eqref{Sec4_biorth1} follows. From \eqref{Sec4_Eqn11},
$$ D_j = -i\kappa_{\alpha,\beta} \varphi_{\mu_j}^{\alpha,\beta}(1) (\rho+i\mu_j) \Psi_{\alpha,\beta}'(\mu_j). $$
Since the zero $\mu_j$ is simple,
$\Psi_{\alpha,\beta}'(\mu_j)\neq0.$
Furthermore, $\varphi_{\mu_j}^{\alpha,\beta}(1)\neq0.$ Indeed, by \eqref{Sec4_Eqn3}, $(\varphi_{\mu_j}^{\alpha,\beta})'(1)=0.$
If also $\varphi_{\mu_j}^{\alpha,\beta}(1)=0$, uniqueness for the Jacobi differential equation with the initial data
$$ \varphi_{\mu_j}^{\alpha,\beta}(1) = (\varphi_{\mu_j}^{\alpha,\beta})'(1)=0 $$
would imply
$$ \varphi_{\mu_j}^{\alpha,\beta}\equiv0, $$
which is impossible because
$$ \varphi_{\mu_j}^{\alpha,\beta}(0)=1. $$
Thus $D_j\neq0$.
Next, \eqref{Sec4_Eqn12_thm} gives
$$ \int_{-1}^{1} G_{\mu_j}^{\alpha,\beta}(-x) A_{\alpha,\beta}(x)\,dx = \kappa_{\alpha,\beta} \Psi_{\alpha,\beta}(\mu_j)=0, $$
which proves \eqref{Sec4_biorth2}.
Finally, applying the bounded linear functional
$$ h\longmapsto \int_{-1}^{1} h(x)G_{\mu_j}^{\alpha,\beta}(-x) A_{\alpha,\beta}(x)\,dx $$
to the Riesz-basis expansion of $g$, we obtain
$$\int_{-1}^{1} g(x)G_{\mu_j}^{\alpha,\beta}(-x) A_{\alpha,\beta}(x)\,dx= c_0 \int_{-1}^{1} G_{\mu_j}^{\alpha,\beta}(-x) A_{\alpha,\beta}(x)\,dx+ \sum_{k\neq0} c_k\mathcal I(\mu_k,\mu_j).$$
By \eqref{Sec4_biorth1} and \eqref{Sec4_biorth2}, the right-hand side reduces to $ c_jD_j.$ Since $D_j\neq0$, we obtain
$$ c_j = \frac1{D_j} \int_{-1}^{1} g(x)\,G_{\mu_j}^{\alpha,\beta}(-x) A_{\alpha,\beta}(x)\,dx, $$
which proves \eqref{Sec4_coeff}.
\end{proof}

\begin{theorem}[Sampling Theorem]
Let $ f\in PW^\theta_{\alpha,\beta}.$
Then
\begin{align}\label{Sec4_Eqn16}
f(\lambda) = f(0)S_0^\theta(\lambda) + \sum_{j\in\mathbb Z\setminus\{0\}} f(s_j)S_j^\theta(\lambda),
\end{align}
where
\begin{align}\label{Sec4_Eqn17}
S_0^\theta(\lambda) = e^{-\frac i2\lambda^2\cot\theta} \frac{ \Psi_{\alpha,\beta}(\lambda\csc\theta) }{ \Psi_{\alpha,\beta}(0) }, 
\end{align}
and, for $j\ne0$,
\begin{align}\label{Sec4_Eqn18}
S_j^\theta(\lambda) = e^{-\frac i2(\lambda^2-s_j^2)\cot\theta} \frac{ \lambda\sin\theta }{ s_j(\lambda-s_j) } \frac{ \Psi_{\alpha,\beta}(\lambda\csc\theta) }{ \Psi_{\alpha,\beta}'(\mu_j) }, \quad \mu_j=s_j\csc\theta. 
\end{align}
The value at $\lambda=s_j$ in \eqref{Sec4_Eqn18} is understood by continuity. The interpolation functions satisfy
\begin{align}\label{Sec4_Eqn19}
S_j^\theta(s_k)=\delta_{jk}, \quad j,k\in\mathbb Z.
\end{align}
The series in \eqref{Sec4_Eqn16} converges in the $PW^\theta_{\alpha,\beta}$-norm and uniformly on compact subsets of $\mathbb R$.
\end{theorem}

\begin{proof}
Let $ F=\mathcal T^\theta_{\alpha,\beta}f. $
Then $\operatorname{supp}F\subset[-1,1].$
Define
\begin{align}\label{Sec4_Eqn20}
g(x) = e^{-\frac i2x^2\cot\theta}F(x). 
\end{align}

Using the definition of the FrOC kernel,
\begin{align} \label{Sec4_Eqn21} 
f(\lambda) &= \mathcal H^\theta_{\alpha,\beta}F(\lambda)= \int_{-1}^{1} F(x) G^\theta_{\alpha,\beta}(-x,\lambda) A_{\alpha,\beta}(x)\,dx \nonumber\\ 
&= e^{-\frac i2\lambda^2\cot\theta} \int_{-1}^{1} g(x) G_{\lambda\csc\theta}^{\alpha,\beta}(-x) A_{\alpha,\beta}(x)\,dx. \end{align}
Let $b=\lambda\csc\theta.$ By Lemma \ref{orthogonal}, $g$ admits a convergent expansion with respect to the Riesz system
$\{1,G_{\mu_j}:j\ne0\}$. For $j\ne0$, its coefficient corresponding to $G_{\mu_j}$ is determined by the biorthogonal functional
$$ \int_{-1}^{1} g(x)G_{\mu_j}(-x)A(x)\,dx. $$
But from \eqref{Sec4_Eqn21},
\begin{align}\label{Sec4_Eqn22}
\int_{-1}^{1} g(x)G_{\mu_j}(-x)A(x)\,dx = e^{\frac i2s_j^2\cot\theta}f(s_j). 
\end{align}
Let $D_j=\mathcal{I}(\mu_j,\mu_j).$
Applying \eqref{Sec4_Eqn22} to the Riesz expansion therefore gives initially
\begin{align}\label{Sec4_Eqn23}
f(\lambda) = c_0 e^{-\frac i2\lambda^2\cot\theta} \kappa_{\alpha,\beta}\Psi(b) + \sum_{j\ne0} f(s_j)R_j^\theta(\lambda),
\end{align}
where
$$ R_j^\theta(\lambda) = e^{-\frac i2(\lambda^2-s_j^2)\cot\theta} \frac{\mathcal{I}(\mu_j,b)}{D_j}. $$
Using \eqref{Sec4_Eqn10} and \eqref{Sec4_Eqn11},
\begin{align}\label{Sec4_Eqn24}
R_j^\theta(\lambda) = e^{-\frac i2(\lambda^2-s_j^2)\cot\theta} \frac{ (\rho+ib)\Psi(b) }{ (b-\mu_j) (\rho+i\mu_j) \Psi'(\mu_j) }. 
\end{align}
The constant coefficient $c_0$ may be eliminated by evaluating \eqref{Sec4_Eqn23} at $\lambda=0$. Since $ \Psi(0)\ne0,$ we obtain
\begin{align}\label{Sec4_Eqn25}
f(\lambda) = f(0)S_0^\theta(\lambda) + \sum_{j\ne0} f(s_j) \left[ R_j^\theta(\lambda) - R_j^\theta(0) S_0^\theta(\lambda) \right], 
\end{align}
where $S_0^\theta$ is given by \eqref{Sec4_Eqn17}. Now a direct simplification gives
$$ R_j^\theta(\lambda) - R_j^\theta(0)S_0^\theta(\lambda) = e^{-\frac i2(\lambda^2-s_j^2)\cot\theta} \frac{ b\Psi(b) }{ \mu_j(b-\mu_j)\Psi'(\mu_j) }. $$
Since $ b=\lambda\csc\theta, \quad \mu_j=s_j\csc\theta,$
we have
$\frac{b}{\mu_j(b-\mu_j)} = \frac{\lambda\sin\theta} {s_j(\lambda-s_j)}.$ This proves \eqref{Sec4_Eqn18}.

For $j\ne0$,
$$ S_j^\theta(0)=0. $$
If $k\ne j$, then
$\Psi(\mu_k)=0$, so
$$ S_j^\theta(s_k)=0. $$
Finally,
$$ \Psi(\lambda\csc\theta) = \Psi'(\mu_j) (\lambda-s_j)\csc\theta + o(\lambda-s_j), $$
and therefore
$$ \lim_{\lambda\to s_j} S_j^\theta(\lambda) = 1. $$
Thus \eqref{Sec4_Eqn19} follows.
The Riesz-basis expansion converges in $L^2((-1,1),A_{\alpha,\beta})$. Since multiplication by the chirp in \eqref{Sec4_Eqn20} is unitary and the FrOC transform is an $L^2$-isomorphism, the sampling series converges in the $PW^\theta_{\alpha,\beta}$-norm.

Moreover, for every compact $K\subset\mathbb R$,
$$ |h(\lambda)| \le \|\mathcal T^\theta h\|_{L^2(\R, A_{\alpha,\beta})} \left( \int_{-1}^{1} |G^\theta_{\alpha,\beta}(-x,\lambda)|^2 A_{\alpha,\beta}(x)\,dx \right)^{1/2}, $$
and the second factor is bounded uniformly for $\lambda\in K$. Hence norm convergence implies uniform convergence on compact subsets.
\end{proof}

\begin{rem}
For
$\theta=\frac{\pi}{2},$
the chirp factors disappear, $s_j=\mu_j$, and the formula becomes
$$ f(\lambda) = f(0) \frac{\Psi_{\alpha,\beta}(\lambda)} {\Psi_{\alpha,\beta}(0)} + \sum_{j\ne0} f(\mu_j) \frac{ \lambda }{ \mu_j(\lambda-\mu_j) } \frac{ \Psi_{\alpha,\beta}(\lambda) }{ \Psi_{\alpha,\beta}'(\mu_j) }. $$
This is the corresponding sampling formula for the non-fractional Opdam–Cherednik setting.
\end{rem}

\section{Landau’s necessary density conditions}\label{Sec5}
In this section, we establish the necessary density conditions for sampling and interpolation. Because the FrOC transform is asymmetric, we present the Landau theory for two distinct bandlimited spaces: the spatially bandlimited space $\mathcal{B}_{\alpha,\beta}^\theta(S)$ and the spectrally bandlimited space $\mathbf{PW}_{\alpha,\beta}^\theta(S)$. Our approach is inspired by Landau's concentration-operator method and its adaptation to the Hankel transform by Abreu and Bandeira \cite{AbreuLandau12}.

\subsection{Density Conditions for the Spatially Bandlimited Space \texorpdfstring{$\mathcal{B}_{\alpha,\beta}^\theta(S)$}{B(S)}}
We first consider functions in the spatial domain whose FrOC transform is compactly supported. Let $S\subset \R$ and define the space
$$\mathcal{B}_{\alpha,\beta}^{\theta}(S)=\{f\in L^2(\R, A_{\alpha,\beta}):\, \operatorname{supp} (\mathcal{H}_{\alpha,\beta}^\theta f)\subseteq S\}.$$
Define the spatial restriction operator on $I=[-R,R]$ by
$$(P_I f)(x)=\mathbf 1_{I}(x)f(x)$$
and let the spectral projection on $S=[-\Omega,\Omega]$ by
$$D^\theta_S f=
\mathcal{T}_{\alpha,\beta}^\theta P_S\H_{\alpha,  \beta}^\theta f.$$
Let this spectral concentration operator be 
\begin{align}\label{Concentration_opt_in_B}
\mathcal{L}_{I,S}^\theta = \mathcal{H}_{\alpha,\beta}^\theta D_S^\theta P_I  \mathcal{T}_{\alpha,\beta}^\theta = P_S \left( \mathcal{H}_{\alpha,\beta}^\theta P_I \mathcal{T}_{\alpha,\beta}^\theta \right) P_S.
\end{align}
The eigenvalues $\lambda_k(I,S)$ are the concentration eigenvalues, and consequently
$$\operatorname{Tr}(\mathcal{L}^\theta_{I,S})
=\sum_k\lambda_k(I,S).$$

We begin the density analysis by determining the leading asymptotic behaviour of the trace of the concentration operator.
\begin{lem}\label{trace_Lemma}
Let $\alpha>-\frac{1}{2}$, $\alpha\ge\beta\ge -\frac{1}{2}$, and
$\rho=\alpha+\beta+1$. Then 
$$\operatorname{Tr}(\mathcal{L}^\theta_{I,S})= \frac{2^{1-2\rho} }{\pi} R\,\Omega \csc{\theta} + \operatorname{O}(1).$$
\end{lem}

\begin{proof}
 The spectral concentration operator $\mathcal{L}^\theta_{I,S}$ acts on a function $f \in L^2(S, \sigma_{\alpha,\beta}(\lambda \csc \theta))$ and it has an integral kernel of the form
$$(\mathcal{L}^\theta_{I,S} f)(\lambda) = \int_{S} \mathcal{K}_R^\theta(\lambda, \eta) f(\eta) d\sigma_{\alpha,\beta}(\eta \csc \theta) \quad \text{for } \lambda \in S$$
where
$$K_{\alpha,\beta}^\theta(\lambda,\eta)= \int_{I} \mathcal{G}_{\alpha,\beta}^{\theta}(x,\lambda)\, \overline{\mathcal{G}_{\alpha,\beta}^{\theta}(x,\eta)} \,A_{\alpha,\beta}(x)\,dx.$$
Consequently,
\begin{align}\label{trace_def}
\operatorname{Tr}(\mathcal{L}_{I,S})&=
\int_{S}K_{\alpha,\beta}^\theta(\lambda,\lambda)
\, d\sigma_{\alpha,\beta}(\lambda \csc \theta)\nonumber\\
&=\int_{S} \left( \int_{I} |\mathcal{G}_{\alpha,\beta}^{\theta}(x,\lambda)|^2 A_{\alpha,\beta}(x) dx \right) d\sigma_{\alpha,\beta}(\lambda \csc \theta).
\end{align}
Let $\mu=\lambda\csc{\theta}$ and $W=\Omega\csc{\theta}$. Since
$$\mathcal{G}_{\alpha,\beta}^{\theta}(x,\lambda) = e^{-\frac{i}{2}(x^2+\lambda^2)\cot\theta} G_{\mu}^{\alpha,\beta}(x),$$
we have
$$I_R(\mu):=\int_{-R}^{R} |\mathcal{G}_{\alpha,\beta}^{\theta}(x,\lambda)|^2 A_{\alpha,\beta}(x)~dx=\int_{-R}^{R} |G_\mu^{\alpha,\beta}(x)|^2 A_{\alpha,\beta}(x) dx.$$
Using the definition $G_\mu^{\alpha,\beta}(x) = \varphi_\mu^{\alpha,\beta}(x) - \frac{1}{\rho-i\mu} \frac{d}{dx}\varphi^{\alpha,\beta}_\mu(x)$, we evaluate 
\begin{align}\label{abs(G_mu)2_calculation}
|G_\mu^{\alpha,\beta}(x)|^2
  &=\left(\varphi_\mu^{\alpha,\beta}(x)\right)^2+\frac{\left(\frac{d}{dx}\varphi_\mu^{\alpha,\beta}(x)\right)^2}{\rho^2+\mu^2}-\frac{2\rho}{\rho^2+\mu^2}\varphi_\mu^{\alpha,\beta}(x)\,\frac{d}{dx}\varphi_\mu^{\alpha,\beta}(x).
  \end{align}
Since $A_{\alpha,\beta}$ and
$\varphi_\mu^{\alpha,\beta}$ are even while $\frac{d}{dx}\varphi_\mu^{\alpha,\beta}$ is odd, the last term is odd. Consequently, its integral over $[-R, R]$ vanishes. Thus
\begin{align}
I_R(\mu)= 2 \int_{0}^{R} \left( \varphi_\mu(x)^2 + \frac{\left(\frac{d}{dx}\varphi_\mu^{\alpha,\beta}(x)\right)^2}{\rho^2+\mu^2} \right) A_{\alpha,\beta}(x) dx
\end{align}
The Jacobi equation has the Sturm–Liouville form
\begin{align}\label{Strum_Liouville}
\left(A_{\alpha,\beta}(x)\,\frac{d}{dx}
\varphi_\lambda^{\alpha,\beta}(x)\right)^\prime&=
-(\lambda^2+\rho^2)
A_{\alpha,\beta}(x)
\varphi_\lambda^{\alpha,\beta}(x).
\end{align}
Using integration by parts, we obtain
\begin{align*}
\int_0^R
A_{\alpha,\beta}(x)\left(\frac{d}{dx}\varphi_\mu^{\alpha,\beta}(x)\right)^2dx&=
A_{\alpha,\beta}(R)
\varphi_\mu^{\alpha,\beta}(R)\,
\left.\frac{d}{dx}\right|_{x=R}\varphi_\mu^{\alpha,\beta}(x)\\
&\hspace{3cm}+(\mu^2+\rho^2)
\int_0^R
A_{\alpha,\beta}(x)
\left(\varphi_\mu^{\alpha,\beta}(x)\right)^2dx.
\end{align*}
Therefore,
\begin{align}\label{I(R)}
I_R(\mu)=
4\int_0^R
A_{\alpha,\beta}(x)
\left(\varphi_\mu^{\alpha,\beta}(x)\right)^2dx
+
\frac{2}
{\rho^2+\mu^2}A_{\alpha,\beta}(R)
\varphi_\mu^{\alpha,\beta}(R)\, 
\left.\frac{d}{dx}\right|_{x=R}\varphi_\mu^{\alpha,\beta}(x).
\end{align}
The OC function $A_{\alpha,\beta}(R) = (\sinh R)^{2\alpha+1}(\cosh R)^{2\beta+1}$ expands as
  \begin{align}\label{Assymptotic A}
  A_{\alpha,\beta}(R) = 2^{-2\rho} e^{2\rho R} (1 + O(e^{-2R})).
  \end{align}
  
  For real $\mu$,
$C_{\alpha,\beta}(-\mu) = \overline{C_{\alpha,\beta}(\mu)}.$
Although $C_{\alpha,\beta}(\mu)$ has a simple pole at $\mu=0$, the pole disappears after multiplication by the Plancherel density. To make this cancellation explicit, define, for $\mu\neq0$,
$$ r_{\alpha,\beta}(\mu) := \frac{C_{\alpha,\beta}(\mu)} {C_{\alpha,\beta}(-\mu)}.$$
Since $C_{\alpha,\beta}$ has a simple pole at the origin,
$$ \lim_{\mu\to0} r_{\alpha,\beta}(\mu)=-1. $$
Thus $r_{\alpha,\beta}$ extends smoothly to $[-W,W]$ by setting $ r_{\alpha,\beta}(0)=-1.$
Moreover,
$$ |r_{\alpha,\beta}(\mu)|=1, \quad \mu\in\mathbb R.$$
Fix $R_0>0$ sufficiently large so that the Harish–Chandra expansion is valid for $x\ge R_0$. The contribution of $[0,R_0]$ to the trace is independent of $R$, and therefore is $O(1)$.

For $x\ge R_0$, the Harish–Chandra expansion gives (see \cite[Theorem 2.11 and Corollary 2.13, pp. 246--247]{EKN_Pusti}), for $R\to \infty$
\begin{align}\label{Harish_chandra varphi}
\varphi_\mu^{\alpha,\beta}(R) = e^{-\rho R} \left( C_{\alpha,\beta}(\mu)e^{i\mu R} + C_{\alpha,\beta}(-\mu)e^{-i\mu R} \right) + O(e^{-(\rho+2)R}).
\end{align}
After division by $|C_{\alpha,\beta}(\mu)|^2$, this yields 
\begin{align}\label{division_term}
\frac{ A_{\alpha,\beta}(x) \varphi_\mu^{\alpha,\beta}(x)^2 }{ |C_{\alpha,\beta}(\mu)|^2 } = 2^{-2\rho} \Big( 2 + r_{\alpha,\beta}(\mu)e^{2i\mu x} + r_{\alpha,\beta}(\mu)^{-1}e^{-2i\mu x} \Big) + O(e^{-2x}),
\end{align}
uniformly for $\mu\in[-W,W]$.
Integrating \eqref{division_term} from $R_0$ to $R$, we obtain
\begin{align}\label{oscillary_term}
\frac{1}{|C_{\alpha,\beta}(\mu)|^2} \int_0^R A_{\alpha,\beta}(x) \varphi_\mu^{\alpha,\beta}(x)^2\,dx ={}& 2^{1-2\rho}R+ 2^{-2\rho}r_{\alpha,\beta}(\mu) \frac{e^{2i\mu R}-1}{2i\mu}\nonumber \\ 
&\hspace{1cm}+ 2^{-2\rho}r_{\alpha,\beta}(\mu)^{-1} \frac{1-e^{-2i\mu R}}{2i\mu} + E_R(\mu), 
\end{align}
where the fixed $R_0$-terms have been absorbed into $E_R$, and
$$ \int_{-W}^{W}|E_R(\mu)|\,d\mu=O(1).$$
We claim that the two oscillatory terms in \eqref{oscillary_term} also contribute only $O(1)$ after integration in $\mu$. Indeed, since $r_{\alpha,\beta}$ is $C^1$ on $[-W,W]$,
$$ r_{\alpha,\beta}(\mu) = r_{\alpha,\beta}(0) + \mu q_{\alpha,\beta}(\mu), $$
where $q_{\alpha,\beta}$ is bounded. Therefore,
\begin{align*} 
\int_{-W}^{W} r_{\alpha,\beta}(\mu) \frac{e^{2i\mu R}-1}{\mu}\,d\mu= r_{\alpha,\beta}(0) \int_{-W}^{W} \frac{e^{2i\mu R}-1}{\mu}\,d\mu + \int_{-W}^{W} q_{\alpha,\beta}(\mu) (e^{2i\mu R}-1)\,d\mu.
\end{align*}
The second integral is $O(1)$, while
\begin{align*}
\int_{-W}^{W} \frac{e^{2i\mu R}-1}{\mu}\,d\mu &= 2i\int_0^W \frac{\sin(2\mu R)}{\mu}\,d\mu = 2i\int_0^{2WR} \frac{\sin t}{t}\,dt = O(1).
\end{align*}
Thus
$$ \int_{-W}^{W} r_{\alpha,\beta}(\mu) \frac{e^{2i\mu R}-1}{\mu}\,d\mu =O(1).$$
The same argument applies to the term containing
$r_{\alpha,\beta}(\mu)^{-1}$. 

Now differentiating the Harish–Chandra expansion \eqref{Harish_chandra varphi} gives 
\begin{align}\label{Derivative Harish_chandra varphi}
\left.\frac{d}{dx}\right|_{x=R}\varphi_\mu^{\alpha,\beta}(x) = e^{-\rho R} \left( (i\mu - \rho)C_{\alpha,\beta}(\mu)e^{i\mu R} + (-i\mu - \rho)C_{\alpha,\beta}(-\mu)e^{-i\mu R} \right) + O(e^{-(\rho+2)R}).
\end{align}
Multiplying \eqref{Assymptotic A}, \eqref{Harish_chandra varphi}, and \eqref{Derivative Harish_chandra varphi}, we obtain
\begin{align}
&\frac{A_{\alpha,\beta}(R)\varphi_\mu^{\alpha,\beta}(R)\left.\frac{d}{dx}\right|_{x=R}\varphi_\mu^{\alpha,\beta}(x)}{\vert{}C_{\alpha,\beta}(\mu)\vert{}^2}\nonumber\\
&=2^{-2\rho} \left[(i\mu - \rho)r_{\alpha,\beta}(\mu)e^{2i\mu R} - (i\mu+\rho)(r_{\alpha,\beta}(-\mu))^{-1}e^{-2i\mu R}-2\rho \right] + O(e^{-2R}).
\end{align}
Since $r_{\alpha,\beta}$ and $r_{\alpha,\beta}^{-1}$ are bounded on $[-W,W]$, it follows that
$$\int_{-W}^{W}\frac{A_{\alpha,\beta}(R)\varphi_\mu^{\alpha,\beta}(R)\,\left.\frac{d}{dx}\right|_{x=R}\varphi_\mu^{\alpha,\beta}(x)}{(\rho^2+\mu^2)\vert{}C_{\alpha,\beta}(\mu)\vert{}^2}=\operatorname{O}(1).$$
Substituting
$$d\sigma_{\alpha, \beta}(\mu) = \frac{d \mu}{8 \pi \vert{}C_{\alpha, \beta}(\mu)\vert{}^2} - \frac{\rho}{i\mu} \frac{d \mu}{8 \pi \vert{}C_{\alpha, \beta}(\mu)\vert{}^2},$$
we obtain from \eqref{trace_def} that
$$\operatorname{Tr}(\mathcal{L}^\theta_{I,S}) = \int_{-\Omega\csc{\theta}}^{\Omega\csc{\theta}} I_R(\mu) \left( \frac{1}{8 \pi \vert{}C_{\alpha, \beta}(\mu)\vert{}^2} - \frac{\rho}{i\mu} \frac{1}{8 \pi \vert{}C_{\alpha, \beta}(\mu)\vert{}^2} \right) d\mu.$$
Since $I_R(\mu)$ and $\vert{}C_{\alpha,\beta}(\mu)\vert{}^2$ are even functions, and $\frac{\rho}{i\mu}$ is a strictly odd function, 
$$\int_{-\Omega\csc{\theta}}^{\Omega\csc{\theta}} I_R(\mu) \left( \frac{\rho}{i\mu} \right) \frac{d\mu}{8 \pi \vert{}C_{\alpha, \beta}(\mu)\vert{}^2} = 0.$$
Hence,
\begin{align}
\operatorname{Tr}(\mathcal{L}^\theta_{I,S})&= 8 \cdot 2^{-2\rho} R\int_{-W}^{W}\left|C_{\alpha,\beta}(\mu)\right|^2~\frac{d\mu}{8\pi\vert{}C_{\alpha,\beta}(\lambda)\vert{}^{2}}+\operatorname{O}(1)\nonumber\\
&= \frac{2^{-2\rho} R}{\pi} \int_{-W}^{W}  d\mu + \operatorname{O}(1)= \frac{2^{1-2\rho}}{\pi} WR+\operatorname{O}(1).
\end{align}
Since $W=\Omega\csc{\theta}$,
$$\operatorname{Tr}(\mathcal{L}^\theta_{I,S})= \frac{2^{1-2\rho} }{\pi} \Omega R\csc{\theta} + \operatorname{O}(1).$$
This completes the proof.
\end{proof}

The trace asymptotic alone is not sufficient for Landau's eigenvalue-counting argument. We therefore estimate the trace defect
$$ \operatorname{Tr}\left(\mathcal L_{I,S}^{\theta}-(\mathcal L_{I,S}^{\theta})^2\right), $$
which measures the transition of the concentration eigenvalues between $0$ and $1$.
\begin{lem}\label{norm_Lemma}
Let $\alpha\ge \beta\ge -\frac{1}{2}$, $\alpha>-\frac{1}{2}$ and
$\rho=\alpha+\beta+1$. Then 
$$\operatorname{Tr} \left( \mathcal L_{I,S}^\theta - (\mathcal L_{I,S}^\theta)^2 \right) = O(\log R),\quad R\to \infty.$$
Consequently,
$$\operatorname{Tr}\big((\mathcal{L}_{I,S}^\theta)^2\big) = \frac{2^{1-2\rho}}{\pi} \Omega R\csc\theta +  O(\log R).$$
\end{lem}
\begin{proof}
Let
$$ Q_S^\theta = \mathcal T_{\alpha,\beta}^{\theta} P_S \mathcal H_{\alpha,\beta}^{\theta}$$
denote the bandlimiting projection onto
$\mathcal B_{\alpha,\beta}^{\theta}(S)$. Now, we have
$$ \operatorname{Tr} \left( \mathcal L_{I,S}^{\theta} - (\mathcal L_{I,S}^{\theta})^2 \right) = \operatorname{Tr} \left( P_IQ_S^\theta P_I - (P_IQ_S^\theta P_I)^2 \right).$$
Since $Q_S^\theta$ is a projection,
\begin{align}
P_IQ_S^\theta P_I - P_IQ_S^\theta P_IQ_S^\theta P_I &= P_IQ_S^\theta (I-P_I) Q_S^\theta P_I \nonumber\\ &= P_IQ_S^\theta P_{I^c} Q_S^\theta P_I.\nonumber
\end{align}
Thus
\begin{align}\label{trace_norm}
\operatorname{Tr} \left( \mathcal L_{I,S}^{\theta} - (\mathcal L_{I,S}^{\theta})^2 \right) = \left\| P_{I^c}Q_S^\theta P_I \right\|_{HS}^2.
\end{align}
Let $\mathcal K_S^\theta(x,y)$ denote the integral kernel of
$Q_S^\theta$. By the inversion formula,
$$ \mathcal K_S^\theta(x,y) = \int_S \mathcal G_{\alpha,\beta}^{\theta}(x,\lambda)\overline{ \mathcal G_{\alpha,\beta}^{\theta}(y,\lambda)} \,d\sigma_{\alpha,\beta}(\lambda\csc\theta).$$
Consequently, \eqref{trace_norm} becomes
$$\operatorname{Tr} \left( \mathcal L_{I,S}^{\theta} - (\mathcal L_{I,S}^{\theta})^2 \right) = \int_I\int_{I^c} |\mathcal K_S^\theta(x,y)|^2 A_{\alpha,\beta}(x) A_{\alpha,\beta}(y) \,dy\,dx .$$

We now estimate the kernel. Using
$$ \mathcal G_{\alpha,\beta}^{\theta}(x,\lambda) = e^{-\frac{i}{2}(x^2+\lambda^2)\cot\theta} G_{\lambda\csc\theta}^{\alpha,\beta}(x), $$
and setting $W=\Omega\csc\theta$, we can write
\begin{align}\label{Kernel_Lem4.2}
\mathcal K_S^\theta(x,y) = e^{-\frac{i}{2}(x^2-y^2)\cot\theta} \int_{-W}^{W} G_{\mu}^{\alpha,\beta}(x) G_{-\mu}^{\alpha,\beta}(y) \,d\sigma_{\alpha,\beta}(\mu).
\end{align}

We now estimate this kernel. For $x>0$, the Harish--Chandra expansion gives, for $\mu\neq0$,
\begin{align}\label{HC_Lem4.2_eqn1}
G_\mu^{\alpha,\beta}(x) = 2C_{\alpha,\beta}(\mu)e^{(-\rho+i\mu)x} + \frac{2\rho}{\rho-i\mu} C_{\alpha,\beta}(-\mu)e^{(-\rho-i\mu)x} + O(e^{-(\rho+2)x}),
\end{align}
uniformly for $\mu$ in compact subsets away from the origin. Similarly,
\begin{align}\label{HC_Lem4.2_eqn2}
G_{-\mu}^{\alpha,\beta}(y) = 2C_{\alpha,\beta}(-\mu)e^{(-\rho-i\mu)y} + \frac{2\rho}{\rho+i\mu} C_{\alpha,\beta}(\mu)e^{(-\rho+i\mu)y} + O(e^{-(\rho+2)y}). 
\end{align}
For negative arguments, the corresponding expansions follow directly from the fact that
$\varphi_\mu^{\alpha,\beta}$ is even,
$(\varphi_\mu^{\alpha,\beta})'$ is odd, and
$$ G_\mu^{\alpha,\beta}(x) = \varphi_\mu^{\alpha,\beta}(x) - \frac{(\varphi_\mu^{\alpha,\beta})'(x)} {\rho-i\mu}. $$
Furthermore,
\begin{align}
 A_{\alpha,\beta}(x)^{1/2} = 2^{-\rho}e^{\rho|x|} \left(1+O(e^{-2|x|})\right), \quad |x|\to\infty. 
 \end{align}
Hence multiplication by
$\sqrt{A_{\alpha,\beta}(x)A_{\alpha,\beta}(y)}$
cancels the factors $e^{-\rho|x|}$ and $e^{-\rho|y|}$ occurring in the Harish--Chandra expansions. We must treat the point $\mu=0$ carefully. From the explicit expression for the Harish--Chandra $C$-function and the expansion of the Gamma function at the origin,
$$ C_{\alpha,\beta}(\mu) = \frac{c_{\alpha,\beta}}{i\mu}+O(1), \quad \mu\to0, $$
with $c_{\alpha,\beta}\neq0$. Hence
$$ C_{\alpha,\beta}(-\mu) = -\frac{c_{\alpha,\beta}}{i\mu}+O(1), $$
and therefore
\begin{align}\label{Lem4.2_eqn3}
\frac{C_{\alpha,\beta}(\mu)} {C_{\alpha,\beta}(-\mu)} \longrightarrow -1, \quad \frac{C_{\alpha,\beta}(-\mu)} {C_{\alpha,\beta}(\mu)} \longrightarrow -1 \quad (\mu\to0).
\end{align}
Both quotients extend smoothly through the origin. Also,
$$ \frac1{|C_{\alpha,\beta}(\mu)|^2} = O(\mu^2), \quad \mu\to0.$$
Substituting \eqref{HC_Lem4.2_eqn1} and \eqref{HC_Lem4.2_eqn2} into \eqref{Kernel_Lem4.2} and using
$$ d\sigma_{\alpha,\beta}(\mu) = \left(1-\frac{\rho}{i\mu}\right) \frac{d\mu}{8\pi|C_{\alpha,\beta}(\mu)|^2}, $$
the four principal oscillatory phases are
$$ e^{i\mu(x-y)},\quad e^{i\mu(x+y)},\quad e^{-i\mu(x+y)},\quad e^{-i\mu(x-y)}. $$
In view of \eqref{Lem4.2_eqn3}, the possible singular parts of their coefficients at $\mu=0$, up to a common constant, combine as
$$ -\frac{\rho}{i\mu}e^{i\mu(x-y)} + \frac{\rho}{i\mu}e^{i\mu(x+y)} + \frac{\rho}{i\mu}e^{-i\mu(x+y)} - \frac{\rho}{i\mu}e^{-i\mu(x-y)}. $$
Their sum is
$$ \frac{2\rho}{i\mu} \left[ \cos\big(\mu(x+y)\big) - \cos\big(\mu(x-y)\big) \right].$$
The above expression extends regularly through $\mu=0$, since the difference of the two cosine terms vanishes to second order at the origin. Moreover, it is odd in $\mu$, so its integral over $[-W,W]$ is zero. Thus the apparent singularities at $\mu=0$ cancel after the four Harish--Chandra terms are combined.

It follows that, in each of the four sign regions determined by $x$ and $y$, the principal part of the normalized kernel is a finite sum of oscillatory integrals of the form
\begin{align}\label{Oscillatory_int1}
 \int_{-W}^{W} a_{\varepsilon,\delta}(\mu) e^{i\mu(\varepsilon x+\delta y)} \,d\mu, \quad \varepsilon,\delta\in\{-1,1\},
 \end{align}
where $\|a_{\varepsilon,\delta}\|_\infty+\|a_{\varepsilon,\delta}^\prime\|_{L^1(-W,W)}<\infty.$ 
For $t\neq0$, integration by parts gives
\begin{align*} \int_{-W}^{W}a_{\varepsilon,\delta}(\mu)e^{i\mu t}\,d\mu &= \frac{ a_{\varepsilon,\delta}(W)e^{iWt} - a_{\varepsilon,\delta}(-W)e^{-iWt} }{it}-\frac1{it} \int_{-W}^{W} a_{\varepsilon,\delta}'(\mu)e^{i\mu t}\,d\mu. \end{align*}
It follows that
$$ \left| \int_{-W}^{W} a_{\varepsilon,\delta}(\mu)e^{i\mu t}\,d\mu \right| \le \frac{C_S}{|t|}. $$
For bounded $t$, the same integral is trivially bounded. Thus
$$ \left| \int_{-W}^{W} a_{\varepsilon,\delta}(\mu) e^{i\mu(\varepsilon x+\delta y)} \,d\mu \right| \le \frac{C_S} {1+|\varepsilon x+\delta y|}.$$
Combining the finitely many terms in \eqref{Oscillatory_int1}, and enlarging $C_S$ to cover the region where $x$ and $y$ remain bounded, we obtain
$$\sqrt{ A_{\alpha,\beta}(x)A_{\alpha,\beta}(y) } \, |\mathcal K_S^\theta(x,y)| \le C_S \left( \frac{1}{1+|x-y|} + \frac{1}{1+|x+y|} \right), \quad x,y\in\mathbb R.$$
Substituting the above expression into the trace-defect identity
$$ \operatorname{Tr} \left( \mathcal L_{I,S}^{\theta} - (\mathcal L_{I,S}^{\theta})^2 \right) = \int_I\int_{I^c} |\mathcal K_S^\theta(x,y)|^2 A_{\alpha,\beta}(x)A_{\alpha,\beta}(y) \,dy\,dx, $$
and using $(a+b)^2\le2a^2+2b^2$, we obtain
$$\operatorname{Tr} \left( \mathcal L_{I,S}^{\theta} - (\mathcal L_{I,S}^{\theta})^2 \right)\le C_S \int_{-R}^{R} \int_{|y|>R} \left[ \frac{1}{(1+|x-y|)^2} + \frac{1}{(1+|x+y|)^2} \right]dy\,dx.$$
Consider first the contribution corresponding to $y>R$ and
$|x-y|$. Since $x\le R<y$,
$$ |x-y|=y-x, $$
and therefore
\begin{align*} 
\int_{-R}^{R}\int_R^\infty \frac{dy\,dx}{(1+|x-y|)^2} &= \int_{-R}^{R}\int_R^\infty \frac{dy\,dx}{(1+y-x)^2} \\ &= \int_{-R}^{R} \frac{dx}{1+R-x}= \log(1+2R).
\end{align*}
The contribution from $y<-R$ is identical. Likewise,
$$ \int_{-R}^{R}\int_{|y|>R} \frac{dy\,dx}{(1+|x+y|)^2} = O(\log R). $$
It follows that
$$\operatorname{Tr} \left( \mathcal L_{I,S}^{\theta} - (\mathcal L_{I,S}^{\theta})^2 \right) = O(\log R).$$
Finally, Lemma \ref{trace_Lemma} gives
$$ \operatorname{Tr} (\mathcal L_{I,S}^{\theta}) = \frac{2^{1-2\rho}}{\pi} \Omega R\csc\theta + O(1). $$
Since
$$ \operatorname{Tr} \big((\mathcal L_{I,S}^{\theta})^2\big) = \operatorname{Tr} (\mathcal L_{I,S}^{\theta}) - \operatorname{Tr} \left( \mathcal L_{I,S}^{\theta} - (\mathcal L_{I,S}^{\theta})^2 \right), $$
we conclude that
$$\operatorname{Tr} \big((\mathcal L_{I,S}^{\theta})^2\big) = \frac{2^{1-2\rho}}{\pi} \Omega R\csc\theta + O(\log R).$$
This completes the proof.
\end{proof}

\begin{proposition}\label{interpolation_prop}
Let $S \subset \mathbb{R}$ be bounded, and $\Lambda = \{\lambda_n\}$ be an interpolation set for $\mathcal{B}_{\alpha,\beta}^{\theta}(S)$. Then $\Lambda$ is uniformly separated; that is, there exists $d>0$ such that 
\[ \inf_{n\neq m}|\lambda_n-\lambda_m|\ge d,\] 
and interpolation is stable.
\end{proposition}
\begin{proof}
For $f\in\mathcal B_{\alpha,\beta}^{\theta}(S)$, the inversion formula gives
$$ f(x) = \int_S \mathcal H_{\alpha,\beta}^{\theta}f(\lambda) \mathcal G_{\alpha,\beta}^{\theta}(x,\lambda) \,d\sigma_{\alpha,\beta}(\lambda\csc\theta). $$
Since $S$ is bounded, point evaluation at every fixed $x\in\mathbb R$ is a continuous functional on
$\mathcal B_{\alpha,\beta}^{\theta}(S)$.
Consider the restriction operator
$$ f\longmapsto \{f(\lambda_n)\}_{n}. $$
By the definition of an interpolation set, this operator maps
$\mathcal B_{\alpha,\beta}^{\theta}(S)$ onto $\ell^2$. Moreover, its graph is closed. Indeed, if $f_k\to f$ in
$\mathcal B_{\alpha,\beta}^{\theta}(S)$ and
$$ \{f_k(\lambda_n)\}_n\longrightarrow \{a_n\}_n \quad\text{in }\ell^2, $$
then, for every fixed $n$, continuity of point evaluation gives
$$ f_k(\lambda_n)\longrightarrow f(\lambda_n). $$
On the other hand, convergence in $\ell^2$ implies coordinatewise convergence, and hence
$$ a_n=f(\lambda_n) $$
for every $n$. Thus the restriction operator has closed graph and, by the Closed Graph Theorem, it is bounded. Since it is also surjective, the Open Mapping Theorem implies that there exists a constant $K_\Lambda>0$ such that, for every
$\{a_n\}\in\ell^2$, there exists
$f\in\mathcal B_{\alpha,\beta}^{\theta}(S)$ satisfying
$$ f(\lambda_n)=a_n,\quad n\in\mathbb Z, $$
and
\begin{align}\label{interpolation_stable}
 \|f\|_{L^2(\R, \,A_{\alpha,\beta})}^2 \le K_\Lambda \sum_n|a_n|^2.
 \end{align}
Thus interpolation is stable.

It remains to prove separation. Since $e^{-\frac{i}{2}x^2\cot\theta}f(x)$ has OC spectral support contained in $ [-\Omega\csc\theta,\Omega\csc\theta]$, the spectral support remains in a fixed compact interval. The Bernstein inequality for OC-bandlimited functions therefore gives a constant $C_S>0$, depending only on $S,\alpha,\beta,\theta$, such that
\begin{align}\label{bern_inq}
\sup_{x\in\mathbb R} \left| \frac{d}{dx} \left( e^{-\frac{i}{2}x^2\cot\theta}f(x) \right) \right| \le C_S\|f\|_{L^2(\R, \, A_{\alpha,\beta})}. 
\end{align}
Suppose that $\Lambda$ is not uniformly separated. Then there exist distinct points $\lambda_n,\lambda_m\in\Lambda$ with
$|\lambda_n-\lambda_m|$ arbitrarily small.
Choose interpolation data supported only at $\lambda_n$, with
$$ a_n=e^{\frac{i}{2}\lambda_n^2\cot\theta}, \quad a_k=0,\quad k\neq n. $$
The $\ell^2$-norm of these data is one. By \eqref{interpolation_stable}, there exists an interpolating function $f$ satisfying
$$ \|f\|_{L^2(\R, \,A_{\alpha,\beta})}\le \sqrt{K_\Lambda}. $$
Moreover,
$$ e^{-\frac{i}{2}\lambda_n^2\cot\theta}f(\lambda_n)=1, $$
whereas
$$ e^{-\frac{i}{2}\lambda_m^2\cot\theta}f(\lambda_m)=0. $$
Hence, by the fundamental theorem of calculus and \eqref{bern_inq},
$$ \begin{aligned} 1 &\le |\lambda_n-\lambda_m| \sup_{x\in[\lambda_n,\lambda_m]} \left| \frac{d}{dx} \left( e^{-\frac{i}{2}x^2\cot\theta}f(x) \right) \right| \\ &\le C_S\sqrt{K_\Lambda}\, |\lambda_n-\lambda_m|. \end{aligned} $$
Therefore
$$ |\lambda_n-\lambda_m| \ge \frac{1}{C_S\sqrt{K_\Lambda}} $$
for every $n\neq m$. Thus $\Lambda$ is uniformly separated.
\end{proof}

\begin{proposition}\label{sampling}
Let $S=[-\Omega,\Omega]$ be fixed and $d>0$. Then there exists $C_{\Omega,d}>0$ such that every $h\in \mathcal{B}_{\alpha,\beta}^{\theta}(S)$ satisfies
\[
|h(\mu)|^2
\le C_{\Omega,d}
\int_{\mu-d}^{\mu+d}|h(\nu)|^2A_{\alpha,\beta}(\nu)\,d\nu,
\quad \mu\in\mathbb R,
\]
where $C_{\Omega,d}$ is independent of $\mu$.
\end{proposition}

\begin{proof}
Let $h\in\mathcal B_{\alpha,\beta}^{\theta}(S)$. Consequently,
$e^{-\frac{i}{2}x^2\cot\theta}h(x) $
is OC-bandlimited to the compact interval
$ [-\Omega\csc\theta,\Omega\csc\theta].$
The local Plancherel–Pólya inequality for OC-bandlimited functions with fixed compact spectral support therefore yields, for every $d>0$,
$$ \left| e^{-\frac{i}{2}\mu^2\cot\theta}h(\mu) \right|^2 \le C_{\Omega,d} \int_{\mu-d}^{\mu+d} \left| e^{-\frac{i}{2}\nu^2\cot\theta}h(\nu) \right|^2 A_{\alpha,\beta}(\nu)\,d\nu, $$
where $C_{\Omega,d}$ is independent of $\mu$.
Therefore,
$$ |h(\mu)|^2 \le C_{\Omega,d} \int_{\mu-d}^{\mu+d} |h(\nu)|^2A_{\alpha,\beta}(\nu)\,d\nu. $$
This proves the result.
\end{proof}

We now relate stable sampling to the concentration eigenvalues. The next lemma provides the eigenvalue estimate required for the lower density bound.
\begin{lem}\label{sampling_eigenvalue}
Let $S\subset \mathbb{R}$ be bounded and let $\Lambda = \{\lambda_n\}$ be a sampling set for the space $\mathcal{B}_{\alpha,\beta}^{\theta}(S)$, whose points are separated by at least $2d > 0$. For $R>0$, let $I = [-R, R]$ and $I^+ = [-R-d, R+d]$, and $n(I^+)$ be the number of points of $\Lambda$ contained in $I^+$. Then 
$$\lambda_{n(I^+)}(I, S) \le \gamma < 1,$$
where $\gamma$ depends on $S$ and $\Lambda$, but independent of $R$.
\end{lem}

\begin{proof}
By the local Plancherel--Pólya estimate for $\mathcal{B}_{\alpha,\beta}^{\theta}(S)$, there exists a constant $C_{S,d} > 0$, independent of $\lambda$, such that
$$\vert{}f(\lambda)\vert{}^2 \le C_{S,d} \int_{\lambda-d}^{\lambda+d} \vert{}f(\eta)\vert{}^2 \, A_{\alpha,\beta}(\eta)\,d\eta, \quad f \in \mathcal{B}_{\alpha,\beta}^{\theta}(S).$$
Let $N = n(I^+)$, and consider the closed subspace
$$E = \{ f \in \mathcal{B}_{\alpha,\beta}^{\theta}(S) : f(\lambda_n) = 0 \text{ for every } \lambda_n \in I^+ \}.$$
Since we impose exactly $N$ homogeneous linear conditions, $\operatorname{codim} E \le N$. Since $\Lambda$ is a sampling set, there exists $A_\Lambda > 0$ such that 
$$A_\Lambda \|f\|^2_{L^2(\R, A_{\alpha,\beta})} \le \sum_{\lambda_n \in \Lambda} |f(\lambda_n)|^2.$$ For any $f \in E$, the evaluations at the nodes within $I^+$ vanish perfectly. Hence,
\begin{align}\label{Sample_local_bound}
A_\Lambda \Vert{}f\Vert{}^2_{L^2(\R, A_{\alpha,\beta})} &\le \sum_{\lambda_n \notin I^+} \vert{}f(\lambda_n)\vert{}^2 \nonumber\\
&\le C_{S,d} \sum_{\lambda_n \notin I^+} \int_{\lambda_n-d}^{\lambda_n+d} \vert{}f(\eta)\vert{}^2 \, A_{\alpha,\beta}(\eta)\,d\eta.
\end{align}
Since the points of $\Lambda$ are separated by at least $2d$, the integration intervals $[\lambda_n-d, \lambda_n+d]$ are mutually disjoint up to endpoints. Furthermore, if $\lambda_n \notin I^+ = [-R-d, R+d]$, then $[\lambda_n-d,\lambda_n+d]\subset\mathbb R\setminus I.$ Consequently, \eqref{Sample_local_bound} becomes
$$A_\Lambda \Vert{}f\Vert{}^2_{L^2(\R, A_{\alpha,\beta})} \le C_{S,d} \int_{\mathbb{R} \setminus I} \vert{}f(\eta)\vert{}^2 \, A_{\alpha,\beta}(\eta)\,d\eta.$$
Since 
$$\|f\|^2_{L^2(\R, A_{\alpha,\beta})}=
\int_{\mathbb R}|f(\eta)|^2\,
A_{\alpha,\beta}(\eta)\,d\eta,$$
we have
$$A_\Lambda \Vert{}f\Vert{}^2_{L^2(\R, A_{\alpha,\beta})} \le C_{S,d} \left( \Vert{}f\Vert{}^2_{L^2(\R, A_{\alpha,\beta})} - \int_I \vert{}f(\eta)\vert{}^2 \, A_{\alpha,\beta}(\eta)\,d\eta\right).$$
Thus
$$\frac{\int_I |f(\eta)|^2 \, A_{\alpha,\beta}(\eta)\,d\eta}{\|f\|^2_{L^2(\R, A_{\alpha,\beta})}} \le 1 - \frac{A_\Lambda}{C_{S,d}} =: \gamma < 1.$$
The constants $A_\Lambda$ and $C_{S,d}$ are independent of $R$, and therefore so is $\gamma$. Since $\mathcal H_{\alpha,\beta}^{\theta}$ is unitary, the spatial concentration operator $ D_S^\theta P_I D_S^\theta $
is unitarily equivalent to $\mathcal L_{I,S}^\theta$ and therefore has the same concentration eigenvalues. For $f\in\mathcal B_{\alpha,\beta}^{\theta}(S)$, $D_S^\theta f=f$, and hence
$$\frac{\langle D_S^\theta P_I D_S^\theta f,f\rangle_{L^2(\R, A_{\alpha,\beta})}}{\|f\|^2_{L^2(\R, A_{\alpha,\beta})}}=
\frac{\displaystyle\int_I|f(\eta)|^2\,A_{\alpha,\beta}(\eta)\,d\eta}
{\|f\|^2_{L^2(\R, A_{\alpha,\beta})}}.$$ 
Since $D_S^\theta P_I D_S^\theta$ and $\mathcal L_{I,S}^\theta$ have the same eigenvalues, the min--max principle gives
$$\lambda_N(I,S)\le\sup_{0\neq f\in E}\frac{\langle D_S^\theta P_I D_S^\theta f,f\rangle_{L^2(\R, A_{\alpha,\beta})}}{\|f\|^2_{L^2(\R, A_{\alpha,\beta})}}\le\gamma.$$
Since $N=n(I^+)$, it follows that
$$\lambda_{n(I^+)}(I,S) \le \gamma < 1.$$
\end{proof}

The interpolation case requires the complementary eigenvalue estimate. The following lemma gives a uniform lower bound for the concentration eigenvalues determined by the interpolation points lying inside the contracted interval.
\begin{lem}\label{interpolation_eigenvalue}
Let $S\subset\mathbb R$ be bounded, and let $\Lambda=\{\lambda_n\}$ be an interpolation set for $\mathcal{B}_{\alpha,\beta}^{\theta}(S)$. Assume that the points of $\Lambda$ are separated by at least $2d>0$. For $R>d>0$, let $I=[-R,R]$, $I^-=[-R+d,R-d],$ and $n(I^-)$ denote the number of points of $\Lambda$ contained in $I^-$. Then 
$$\lambda_{n(I^-)-1}(I,S)\ge\delta>0,$$
where $\delta>0$ depends on $S$ and $\Lambda$ but is independent of $R$.
\end{lem}

\begin{proof}
By the stability of interpolation, there exists a constant $K_\Lambda>0$ and a bounded interpolation operator $T:\ell^2(\Lambda) \longmapsto \mathcal{B}_{\alpha,\beta}^{\theta}(S)$ such that $(Ta)(\lambda_n)=a_n$ for all $n\in\mathbb Z$, and
\begin{align}\label{interpolation_opt}
\Vert{}Ta\Vert{}^2_{L^2(\R, A_{\alpha,\beta})} \le K_\Lambda\sum_n\vert{}a_n\vert{}^2.
\end{align}
Let $N=n(I^-)$ and denote by $\ell^2(I^-)$ the subspace of sequences supported on the points of $\Lambda$ $I^-$. Define 
$$E=T\bigl(\ell^2(I^-)\bigr).$$
Since $T$ interpolates the data, it is injective on $\ell^2(I^-)$. Thus, $\dim E=N$. 

For every $f\in E$, the interpolation data vanish outside $I^-$; hence $f(\lambda_n)=0$ for all $\lambda_n\notin I^-$. By \eqref{interpolation_opt}, 
\begin{align}\label{bound}
\Vert{}f\Vert{}^2_{L^2(\R, A_{\alpha,\beta})} \le K_\Lambda \sum_{\lambda_n\in I^-} \vert{}f(\lambda_n)\vert{}^2.
\end{align}
By the local Plancherel–Pólya estimate, there exists $C_{S,d}>0$, independent of $\lambda$, such that
$$\vert{}f(\lambda)\vert{}^2\le C_{S,d} \int_{\lambda-d}^{\lambda+d} \vert{}f(\eta)\vert{}^2 \, A_{\alpha,\beta}(\eta)\,d\eta.$$
Therefore, \eqref{bound} gives
\begin{align}\label{bound1}
\Vert{}f\Vert{}^2_{L^2(\R, A_{\alpha,\beta})} \le K_\Lambda C_{S,d} \sum_{\lambda_n\in I^-} \int_{\lambda_n-d}^{\lambda_n+d} \vert{}f(\eta)\vert{}^2 \, A_{\alpha,\beta}(\eta)\,d\eta.
\end{align}
Since the points of $\Lambda$ are separated by at least $2d$, the integration intervals $[\lambda_n-d,\lambda_n+d]$ are mutually disjoint up to endpoints. Furthermore, for every $\lambda_n\in I^-=[-R+d,R-d]$, the geometric condition guarantees $[\lambda_n-d,\lambda_n+d]\subset I$. Consequently, \eqref{bound1} yields
$$\Vert{}f\Vert{}^2_{L^2(\R, A_{\alpha,\beta})} \le K_\Lambda C_{S,d} \int_I \vert{}f(\eta)\vert{}^2 \, A_{\alpha,\beta}(\eta)\,d\eta.$$
Hence, for $0\ne f\in E$,
\begin{align}\label{lower_bound2}
\frac{\int_I\vert{}f(\eta)\vert{}^2 \, A_{\alpha,\beta}(\eta)\,d\eta}{\Vert{}f\Vert{}^2_{L^2(\R, A_{\alpha,\beta})}} \ge \frac{1}{K_\Lambda C_{S,d}} =:\delta>0.
\end{align}
The constants $K_\Lambda$ and $C_{S,d}$ are strictly independent of $R$ and therefore so is $\delta$. Similarly, using the same argument as in Lemma \ref{sampling_eigenvalue}, we obtain
$$\lambda_{N-1}(I,S) \ge \inf_{0\neq f\in E} \frac{\langle D_S^\theta P_I D_S^\theta f,f\rangle_{L^2(\R, A_{\alpha,\beta})}}{\|f\|^2_{L^2(\R, A_{\alpha,\beta})}} \ge \delta.$$
Substituting $N=n(I^-)$ concludes $\lambda_{n(I^-)-1}(I,S)\ge\delta>0$.
\end{proof}


\begin{theorem}[\textbf{Sampling}]
Let $S=[-\Omega,\Omega]$. If $\Lambda = \{\lambda_n\}$ is a sampling set for $\mathcal{B}_{\alpha,\beta}^\theta(S)$ with separation constant $2d>0$, then 
$$\liminf_{R \to \infty} \frac{n([-R,R])}{R} \ge \frac{2^{1-2\rho}}{\pi} \Omega \csc \theta.$$
\end{theorem}
\begin{proof}
Let $I = [-R, R]$, $I^+ = [-R-d, R+d]$, and $N = n(I^+)$. By Lemma \ref{sampling_eigenvalue}, we have 
\begin{align}\label{sampling_eigenvalue_result}
\lambda_N(I, S) \le \gamma < 1,
\end{align}
where $\gamma$ is independent of $R$.
Choose $\delta$ such that $\gamma < \delta < 1$ and define 
$$ N_\delta(I) = \#\{k : \lambda_k(R) > \delta\}.$$
Therefore, \eqref{sampling_eigenvalue_result} implies 
\begin{align}\label{bound3}
N_\delta(I) \le n(I^+).
\end{align}
Indeed, if more than $n(I^+)$ eigenvalues were larger than $\delta$, then $\lambda_{n(I^+)}(I,S) > \delta > \gamma$, directly contradicting \eqref{sampling_eigenvalue_result}. 

We now estimate $N_\delta(I)$ from below. Since $0 \le \lambda_k(I,S) \le 1$, we write
$$\operatorname{Tr}(\mathcal{L}_{I,S}^{\theta}) = \sum_{\lambda_k > \delta} \lambda_k + \sum_{\lambda_k \le \delta} \lambda_k.$$
The first sum satisfies $\sum_{\lambda_k > \delta} \lambda_k \le N_\delta(I)$. For the second sum, $\lambda_k \le \delta$ implies $1 - \lambda_k \ge 1 - \delta$, and hence
$$\lambda_k \le \frac{\lambda_k(1 - \lambda_k)}{1 - \delta}.$$
Therefore,
\begin{align}
\operatorname{Tr}(\mathcal{L}_{I,S}^{\theta}) &\le N_\delta(I) + \frac{1}{1 - \delta} \sum_k \lambda_k(I,S)(1 - \lambda_k(I,S)) \nonumber\\ 
&= N_\delta(I) + \frac{1}{1 - \delta} \left[ \operatorname{Tr}(\mathcal{L}_{I,S}^{\theta}) - \operatorname{Tr}\big((\mathcal{L}_{I,S}^{\theta})^2\big) \right]. 
\end{align}
Consequently, we obtain from Lemmas \ref{trace_Lemma} and \ref{norm_Lemma} that
\begin{align}\label{bound4}
N_\delta(I) &\ge \operatorname{Tr}(\mathcal{L}_{I,S}^{\theta}) - \frac{1}{1 - \delta} \left[ \operatorname{Tr}(\mathcal{L}_{I,S}^{\theta}) - \operatorname{Tr}\big((\mathcal{L}_{I,S}^{\theta})^2\big) \right]\nonumber\\
&\ge \frac{2^{1-2\rho}}{\pi} \Omega R \csc \theta - \operatorname{O}(\log R). 
\end{align}
Combining \eqref{bound3} and \eqref{bound4} yields 
$$n(I^+) \ge \frac{2^{1-2\rho}}{\pi} \Omega R \csc \theta - \operatorname{O}(\log R). $$
Replacing $R$ by $R-d$ shifts the counting window to $I$, yielding
\begin{align}\label{bound5}
n(I) &\ge \frac{2^{1-2\rho}}{\pi} \Omega (R - d) \csc \theta - \operatorname{O}(\log R)\nonumber\\
&\ge \frac{2^{1-2\rho}}{\pi} \Omega R \csc \theta - \operatorname{O}(\log R).
\end{align}
Finally, we obtain from \eqref{bound5} that
$$\liminf_{R \to \infty} \frac{n(I)}{R} \ge \frac{2^{1-2\rho}}{\pi} \Omega \csc \theta.$$
\end{proof}

\begin{theorem}[\textbf{Interpolation}]
Let $S=[-\Omega,\Omega]$. If $\Lambda = \{\lambda_n\}$ is an interpolation set for $\mathcal{B}_{\alpha,\beta}^\theta(S)$, then 
$$\limsup_{R \to \infty} \frac{n([-R,R])}{R} \le \frac{2^{1-2\rho}}{\pi} \Omega \csc \theta.$$
\end{theorem}

\begin{proof}
Let $I = [-R, R]$, $I^- = [-R+d, R-d]$,  and $N = n(I^-)$. By Lemma \ref{interpolation_eigenvalue}, 
\begin{align}\label{interpolation_eigenvalue_result}
\lambda_{N-1}(I, S) \ge \delta.
\end{align}
Let 
$$N_\delta(I) = \#\{k : \lambda_k(I,S) \ge \delta\}.$$
Since the eigenvalues are arranged in non-increasing order, \eqref{interpolation_eigenvalue_result} implies
\begin{align}\label{bound_interp1}
n(I^-) \le N_\delta(I).
\end{align}

We estimate $N_\delta(I)$ from above. Since $0\le \lambda_k\le 1$, for every eigenvalue satisfying $\lambda_k\ge \delta$
\begin{align*}
1-\lambda_k
&\le
\frac{1}{\delta}\lambda_k(1-\lambda_k)\\
1
&\le
\lambda_k+
\frac{1}{\delta}\lambda_k(1-\lambda_k).
\end{align*}
Summing over all $k$ such that $\lambda_k\ge\delta$, we obtain
\begin{align}\label{bound6_interpolation}
N_\delta(I)
&\le
\sum_{\lambda_k\ge\delta}\lambda_k+\frac{1}{\delta}
\sum_{\lambda_k\ge\delta}
\lambda_k(1-\lambda_k)
\nonumber\\
&\le\operatorname{Tr}(\mathcal L_{I,S}^{\theta})
+\frac{1}{\delta}
\left[\operatorname{Tr}(\mathcal L_{I,S}^{\theta})-
\operatorname{Tr}\big((\mathcal L_{I,S}^{\theta})^2\big)
\right].
\end{align}
Using Lemmas \ref{trace_Lemma} and \ref{norm_Lemma}, we obtain from \eqref{bound6_interpolation} that
\begin{align}
N_\delta(I)
\le\frac{2^{1-2\rho}}{\pi}\Omega R\csc\theta+O(\log R).
\end{align}
Combining it with \eqref{interpolation_eigenvalue_result}, 
we obtain
$$n(I^-)
\le\frac{2^{1-2\rho}}{\pi}\Omega R\csc\theta
+O(\log R).$$
Replacing $R$ by $R+d$ shifts the counting window outward from $I^-$ to $I$, yielding
\begin{align}\label{bound_interp3}
n(I) &\le \frac{2^{1-2\rho}}{\pi} \Omega (R + d) \csc \theta + \operatorname{O}(\log R)\nonumber\\
&\le \frac{2^{1-2\rho}}{\pi} \Omega R \csc \theta + \operatorname{O}(\log R).
\end{align}
Finally, we obtain from \eqref{bound_interp3} that
$$\limsup_{R \to \infty} \frac{n(I)}{R} \le \frac{2^{1-2\rho}}{\pi} \Omega \csc \theta.$$
\end{proof}

\subsection{Density Conditions for the Spectrally Bandlimited Space \texorpdfstring{$PW_{\alpha,\beta}^\theta(S)$}{PW(S)}}
We now return to the Paley--Wiener space introduced in Section $\ref{sec4}$. Let
$ S=[-\Omega,\Omega], \quad I=[-R,R].$
No new concentration construction is needed. Indeed, the concentration of
$f\in PW_{\alpha,\beta}^{\theta}(S)$ on the spectral interval $I$ is governed, after applying the inverse FrOC transform, by
$$ P_S\left(\mathcal T_{\alpha,\beta}^{\theta} P_I\mathcal H_{\alpha,\beta}^{\theta}\right)P_S. $$
The nonzero eigenvalues of this operator coincide with those of
$$ P_I \left(\mathcal H_{\alpha,\beta}^{\theta} P_S\mathcal T_{\alpha,\beta}^{\theta}\right)P_I, $$
since the two operators are respectively of the form $A^*A$ and $AA^*$. Thus the finite-window spectral duality established in the preceding subsection remains valid. However, the asymptotic regime is now different: $S$ is fixed while the spectral interval $I=[-R,R]$ expands. Consequently, the corresponding trace asymptotic must be determined separately.

Let $W=R\csc{\theta}$. The diagonal trace calculation reduces to integrals involving the Jacobi function
$\varphi_\mu^{\alpha,\beta}$. We first record the asymptotic estimate needed below. Here
\begin{align}
K_{\alpha,\beta}^\theta(x,y)
&=\int_{-W}^{W}
\mathcal G_{\alpha,\beta}^\theta(x,\mu)
\overline{\mathcal G_{\alpha,\beta}^\theta(y,\mu)}
\,d\sigma_{\alpha,\beta}(\mu)\nonumber\\
&=e^{-\frac{i}{2}(x^2-y^2)\cot{\theta}}\int_{-W}^{W}
G^{\alpha,\beta}_\mu(x)
\overline{G^{\alpha,\beta}_\mu(y)}
\,d\sigma_{\alpha,\beta}(\mu),
\end{align}
where
$$d\sigma_{\alpha,\beta}(\mu)=\left(1-\frac{\rho}{i\mu}\right)
\frac{d\mu}{8\pi |C_{\alpha,\beta}(\mu)|^2}.$$

We now turn to the spectrally bandlimited Paley--Wiener space. The corresponding trace asymptotic requires uniform high-frequency estimates for the Jacobi function on a fixed spatial interval, which we record first.
\begin{lem}\label{Assymptotic_for_PW}
Assume
$ -\frac12<\alpha<0, \quad -\frac12\le \beta\le\alpha, $
and fix $\Omega>0$. Then there exists $C_\Omega>0$ such that, for every $\mu\ge1$ and $0\le x\le\Omega$,
\begin{align}\label{Lem_5.9_A}
|\varphi_\mu^{\alpha,\beta}(x)| \le C_\Omega(1+\mu x)^{-\alpha-\frac12}.
\end{align}
Moreover, for every $0<\delta<\Omega$,
\begin{align}\label{Lem_5.9_B}
A_{\alpha,\beta}(x)^{1/2} \varphi_\mu^{\alpha,\beta}(x) = 2^\alpha\Gamma(\alpha+1)\mu^{-\alpha} x^{1/2}J_\alpha(\mu x) + O_{\delta,\Omega} \left(\mu^{-\alpha-\frac32}\right)
\end{align}
uniformly for $\delta\le x\le \Omega$.
\end{lem}

\begin{proof}
The Jacobi function satisfies
$$ \varphi_\mu''(x) + \left[ (2\alpha+1)\coth x + (2\beta+1)\tanh x \right]\varphi_\mu'(x) + (\mu^2+\rho^2)\varphi_\mu(x)=0. $$
Put
$$ u_\mu(x) = A_{\alpha,\beta}(x)^{1/2} \varphi_\mu^{\alpha,\beta}(x). $$
Since
$$ \frac{A_{\alpha,\beta}'(x)}{A_{\alpha,\beta}(x)} = (2\alpha+1)\coth x+(2\beta+1)\tanh x, $$
the standard elimination of the first derivative gives
\begin{align}\label{eqn1_Lem_5.9(1)}
 u_\mu''(x) + \left[ \mu^2 + \frac{\frac14-\alpha^2}{\sinh^2x} + \frac{\beta^2-\frac14}{\cosh^2x} \right]u_\mu(x)=0.
 \end{align}
Separate the singular Bessel part by writing
\begin{align}\label{eqn1_Lem_5.9(2)}
 u_\mu''(x) + \left[ \mu^2+\frac{\frac14-\alpha^2}{x^2} \right]u_\mu(x) = -Q_{\alpha,\beta}(x)u_\mu(x),
 \end{align}
where
\begin{align}\label{eqn1_Lem_5.9(3)}
 Q_{\alpha,\beta}(x) = \left(\frac14-\alpha^2\right) \left( \frac1{\sinh^2x}-\frac1{x^2} \right) + \left(\beta^2-\frac14\right)\frac1{\cosh^2x}.
 \end{align}
Now
$$\frac1{\sinh^2x} = \frac1{x^2}-\frac13+O(x^2), \quad x\to0,$$
so $Q_{\alpha,\beta}$ extends continuously to $x=0$. Therefore
\begin{align}\label{eqn1_Lem_5.9(4)}
\sup_{0\le x\le\Omega}|Q_{\alpha,\beta}(x)| <\infty.
\end{align}
Set
$$ s=\mu x, \quad U_\mu(s) = \mu^\nu u_\mu(s/\mu), \quad \nu=\alpha+\frac12. $$

Equation \eqref{eqn1_Lem_5.9(2)} becomes
\begin{align}\label{eqn1_Lem_5.9(5)}
 U_\mu''(s) + \left[ 1+\frac{\frac14-\alpha^2}{s^2} \right]U_\mu(s) = -\frac1{\mu^2} Q_{\alpha,\beta}(s/\mu)U_\mu(s).
 \end{align}
Since
$$ A_{\alpha,\beta}(x)^{1/2} \sim x^{\alpha+1/2} =x^\nu \text{ and } \varphi_\mu(0)=1, $$
we have
\begin{align}\label{eqn1_Lem_5.9(6)}
U_\mu(s)\sim s^\nu, \quad s\downarrow0.
\end{align}
The regular solution of the unperturbed equation
$$ V''+ \left[ 1+\frac{\frac14-\alpha^2}{s^2} \right]V=0 $$
with exactly this normalization is
\begin{align}\label{eqn1_Lem_5.9(7)}
V_\alpha(s) = 2^\alpha\Gamma(\alpha+1)\sqrt{s}\,J_\alpha(s), 
\end{align}
since
$$ J_\alpha(s) = \frac1{\Gamma(\alpha+1)} \left(\frac{s}{2}\right)^\alpha (1+O(s^2)). $$
Take
$$ f_\alpha(s)=\sqrt{s}J_\alpha(s), \quad g_\alpha(s)=\sqrt{s}Y_\alpha(s). $$

Their Wronskian is constant
$$ W(f_\alpha,g_\alpha)=\frac2\pi. $$
Hence the Green kernel of the Bessel equation is
\begin{align}\label{eqn1_Lem_5.9(8)}
 K_\alpha(s,t) = \frac{\pi}{2} \left[ f_\alpha(t)g_\alpha(s) - g_\alpha(t)f_\alpha(s) \right], \quad 0<t\le s. 
 \end{align}
Thus $U_\mu$ satisfies the Volterra equation
\begin{align}\label{eqn1_Lem_5.9(9)}
U_\mu(s) = V_\alpha(s) - \frac1{\mu^2} \int_0^s K_\alpha(s,t) Q_{\alpha,\beta}(t/\mu) U_\mu(t)\,dt.
\end{align}
Now define
$$ m(s) = \left(\frac{s}{1+s}\right)^\nu. $$
Since $0<\nu<\frac12$, the elementary Bessel estimates imply
\begin{align}\label{eqn1_Lem_5.9(10)}
|\sqrt{s}J_\alpha(s)|+ |\sqrt{s}Y_\alpha(s)| \le C_\alpha m(s), \quad s>0. 
\end{align}
Indeed, for $0<s\le1$,
$$ \sqrt{s}J_\alpha(s)=O(s^\nu), $$
while, because $-\frac12<\alpha<0$,
$$ \sqrt{s}Y_\alpha(s) = O(s^\nu)+O(s^{1-\nu}) = O(s^\nu). $$
For $s\ge1$, both quantities are $O(1)$.
Consequently, from \eqref{eqn1_Lem_5.9(8)},
\begin{align}\label{eqn1_Lem_5.9(11)}
|K_\alpha(s,t)| \le C_\alpha m(s)m(t) \le C_\alpha\frac{m(s)}{m(t)}, \quad 0<t\le s. 
\end{align}
Also
\begin{align}\label{eqn1_Lem_5.9(12)}
 |V_\alpha(s)|\le C_\alpha m(s).
 \end{align}
Divide \eqref{eqn1_Lem_5.9(9)} by $m(s)$. Using \eqref{eqn1_Lem_5.9(4)}, \eqref{eqn1_Lem_5.9(11)}, and \eqref{eqn1_Lem_5.9(12)},
$$ \frac{|U_\mu(s)|}{m(s)} \le C+ \frac{C_\Omega}{\mu^2} \int_0^s \frac{|U_\mu(t)|}{m(t)}\,dt. $$
Gronwall's inequality yields
$$ \frac{|U_\mu(s)|}{m(s)} \le C \exp\left(\frac{C_\Omega s}{\mu^2}\right). $$
Since $0\le s\le\mu\Omega$,
$$ \frac{s}{\mu^2}\le\frac{\Omega}{\mu}\le\Omega, $$
and therefore
\begin{align}\label{eqn1_Lem_5.9(13)}
|U_\mu(s)|\le C_\Omega m(s), \quad 0\le s\le\mu\Omega.
\end{align}
Returning to $x$,
$$ |u_\mu(x)| = \mu^{-\nu}|U_\mu(\mu x)| \le C_\Omega \mu^{-\nu} \left( \frac{\mu x}{1+\mu x} \right)^\nu, $$
hence
\begin{align}\label{eqn1_Lem_5.9(14)}
|u_\mu(x)| \le C_\Omega \frac{x^\nu}{(1+\mu x)^\nu}.
\end{align}
On the fixed interval $0\le x\le\Omega$,
$$ A_{\alpha,\beta}(x)^{1/2}\asymp_\Omega x^\nu. $$
Since $u_\mu=A^{1/2}\varphi_\mu$, we obtain
$$ |\varphi_\mu^{\alpha,\beta}(x)| \le C_\Omega (1+\mu x)^{-\nu}, $$
which is precisely \eqref{Lem_5.9_A}.
For the sharper estimate away from $x=0$, subtract $V_\alpha$ in \eqref{eqn1_Lem_5.9(9)}. From \eqref{eqn1_Lem_5.9(11)} and \eqref{eqn1_Lem_5.9(13)},
\begin{align}\label{eqn1_Lem_5.9(15)}
 |U_\mu(s)-V_\alpha(s)| \le \frac{C_\Omega}{\mu^2} m(s)\int_0^s \frac{|U_\mu(t)|}{m(t)}dt \le C_\Omega\frac{s}{\mu^2}m(s).
 \end{align}
If $x\in[\delta,\Omega]$, then $s=\mu x\ge\mu\delta$. Thus, for large $\mu$, $m(s)\asymp1$, and since $s\le\mu\Omega$,
$$ U_\mu(s)-V_\alpha(s) = O_{\delta,\Omega}(\mu^{-1}). $$
Multiplication by $\mu^{-\nu}$ gives
$$ u_\mu(x) = \mu^{-\nu}V_\alpha(\mu x) + O_{\delta,\Omega}(\mu^{-\nu-1}). $$
Using \eqref{eqn1_Lem_5.9(7)},
$$ \mu^{-\nu}V_\alpha(\mu x) = 2^\alpha\Gamma(\alpha+1) \mu^{-\alpha}x^{1/2}J_\alpha(\mu x), $$
while
$$ \mu^{-\nu-1} = \mu^{-\alpha-\frac32}. $$
Therefore,
$$ A_{\alpha,\beta}(x)^{1/2} \varphi_\mu^{\alpha,\beta}(x) = 2^\alpha\Gamma(\alpha+1) \mu^{-\alpha}x^{1/2}J_\alpha(\mu x) + O_{\delta,\Omega} (\mu^{-\alpha-\frac32}), $$
which proves \eqref{Lem_5.9_B}.
\end{proof}

Using the preceding Jacobi-function estimates, we next determine the leading asymptotic behaviour of the integral appearing in the diagonal trace calculation.
\begin{lem}\label{jacobi_weyl}
Assume $\alpha\geq\beta\geq-\frac12$, with $\alpha>-\frac12$. For any fixed $\Omega>0$, 
$$\int_0^\Omega \left( \int_0^W \vert{}\varphi_{\mu}^{\alpha,\beta}(x)\vert{}^2 \frac{d\mu}{2\pi\vert{}C_{\alpha,\beta}(\mu)\vert{}^2} \right) A_{\alpha,\beta}(x)\,dx = \frac{2^{-2\rho}}{\pi}\Omega W + o(W), \quad W \to \infty.$$
\end{lem}

\begin{proof}
Let
$$J_\Omega(W) = \int_0^\Omega \left( \int_0^W |\varphi_\mu^{\alpha,\beta}(x)|^2 \frac{d\mu}{2\pi |C_{\alpha,\beta}(\mu)|^2} \right) A_{\alpha,\beta}(x)\,dx.$$
The contribution of $0\leq\mu\leq1$ is bounded independently of $W$. Hence it is enough to consider $\mu\geq1$.

Fix $0<\delta<\Omega$, and write
$J_\Omega(W) = I_{\delta}^{(1)}(W) + J_\omega{\delta}^{(2)}(W)$, where
$$J_{\delta}^{(1)}(W) = \int_\delta^\Omega \left( \int_0^W |\varphi_\mu^{\alpha,\beta}(x)|^2 \frac{d\mu}{2\pi |C_{\alpha,\beta}(\mu)|^2} \right) A_{\alpha,\beta}(x)\,dx$$
and
$$J_{\delta}^{(2)}(W) = \int_0^\delta \left( \int_0^W |\varphi_\mu^{\alpha,\beta}(x)|^2 \frac{d\mu}{2\pi \vert{}C_{\alpha,\beta}(\mu)\vert{}^2} \right) A_{\alpha,\beta}(x)\,dx.$$

Using equation (5.2) in \cite{Wong_and_Wang}, we obtain that for $\alpha\ge 0$
\begin{align}\label{Lem_5.10_eqn1}
A_{\alpha,\beta}(x)^{1/2} \varphi_\mu^{\alpha,\beta}(x) = 2^\alpha\Gamma(\alpha+1) \mu^{-\alpha}x^{1/2}J_\alpha(\mu x) + O_{\delta,\Omega} (\mu^{-\alpha-\frac32}).
\end{align}
From Lemma \ref{Assymptotic_for_PW}, we obtain \eqref{Lem_5.10_eqn1} also holds for $\alpha>-\frac{1}{2}$.

Now, Stirling's formula yields
\begin{align}\label{Lem_5.10_eqn2}
 |C_{\alpha,\beta}(\mu)|^2 = \frac{ 2^{2\rho+2\alpha-1} \Gamma(\alpha+1)^2 }{\pi} \mu^{-2\alpha-1} \left(1+O(\mu^{-1})\right). 
 \end{align}
Therefore
\begin{align}\label{Lem_5.10_eqn3}
\frac1{2\pi|C_{\alpha,\beta}(\mu)|^2} = \frac{ \mu^{2\alpha+1} }{ 2^{2\rho+2\alpha}\Gamma(\alpha+1)^2 } \left(1+O(\mu^{-1})\right). 
\end{align}
Combining \eqref{Lem_5.10_eqn1} with \eqref{Lem_5.10_eqn3}, uniformly for $x\in[\delta,\Omega]$,
\begin{align}\label{Lem_5.10_eqn4}
A_{\alpha,\beta}(x) |\varphi_\mu^{\alpha,\beta}(x)|^2 \frac1{2\pi|C_{\alpha,\beta}(\mu)|^2} = 2^{-2\rho} \mu xJ_\alpha(\mu x)^2 + O_{\delta,\Omega}(\mu^{-1}).
\end{align}
Hence
\begin{align}\label{Lem_5.10_eqn5}
J_{\delta,\Omega}(W) = 2^{-2\rho} \int_1^W \mu \left( \int_\delta^\Omega xJ_\alpha(\mu x)^2\,dx \right)d\mu + O_{\delta,\Omega}(\log W). 
\end{align}
Using the identity
\begin{align}\label{Lem_5.10_eqn6}
 \int_0^a xJ_\alpha(\mu x)^2\,dx = \frac{a^2}{2} \left[ J_\alpha(\mu a)^2 - J_{\alpha-1}(\mu a) J_{\alpha+1}(\mu a) \right], 
 \end{align}
and the classical large-argument Bessel asymptotics,
$$ J_\alpha(z) = \sqrt{\frac2{\pi z}} \cos\left( z-\frac{\pi\alpha}{2}-\frac{\pi}{4} \right) + O(z^{-3/2}), $$
one obtains
$$ J_\alpha(z)^2 - J_{\alpha-1}(z)J_{\alpha+1}(z) = \frac{2}{\pi z}+O(z^{-2}). $$
Therefore, for fixed $a>0$,
\begin{align}\label{Lem_5.10_eqn7}
\int_0^a xJ_\alpha(\mu x)^2\,dx = \frac{a}{\pi\mu} + O_a(\mu^{-2}),
\end{align}
and hence
\begin{align}\label{Lem_5.10_eqn8}
\int_\delta^\Omega xJ_\alpha(\mu x)^2\,dx = \frac{\Omega-\delta}{\pi\mu} + O_{\delta,\Omega}(\mu^{-2}).
\end{align}
Substitution in \eqref{Lem_5.10_eqn5} gives
\begin{align}\label{Lem_5.10_eqn9}
J_{\delta,\Omega}(W) = \frac{2^{-2\rho}}{\pi} (\Omega-\delta)W + O_{\delta,\Omega}(\log W).
\end{align}
It remains to control $J_{0,\delta}(W)$.
For $-\frac12<\alpha<0$, estimate \eqref{Lem_5.9_A}, together with
$$ A_{\alpha,\beta}(x) \le C_\Omega x^{2\alpha+1} $$
and \eqref{Lem_5.10_eqn3}, yields
$$ A_{\alpha,\beta}(x) |\varphi_\mu^{\alpha,\beta}(x)|^2 \frac1{2\pi|C_{\alpha,\beta}(\mu)|^2} \le C_\Omega \left( \frac{\mu x}{1+\mu x} \right)^{2\alpha+1}. $$
Since $2\alpha+1>0$,
$$ A_{\alpha,\beta}(x) |\varphi_\mu^{\alpha,\beta}(x)|^2 \frac1{2\pi|C_{\alpha,\beta}(\mu)|^2} \le C_\Omega.$$

For $\alpha\ge0$, the uniform Wong–Wang \cite{Wong_and_Wang} expansion together with its remainder estimate in (3.5) for $n=0$ gives the same bound on every fixed interval $0\le x\le\Omega$. Their expansion is uniform in the spatial variable and is accompanied by explicit error control. Thus, for every $\alpha>-\frac12$,
\begin{align}\label{Lem_5.10_eqn10}
0\le J_{0,\delta}(W) \le C_\Omega\delta W.
\end{align}
Since all integrands are nonnegative, $J_\Omega(W)\ge J_{\delta,\Omega}(W)$, and hence
$$ \liminf_{W\to\infty} \frac{J_\Omega(W)}{W} \ge \frac{2^{-2\rho}}{\pi}(\Omega-\delta), $$
whereas \eqref{Lem_5.10_eqn9} and \eqref{Lem_5.10_eqn10} give
$$ \limsup_{W\to\infty} \frac{J_\Omega(W)}{W} \le \frac{2^{-2\rho}}{\pi}(\Omega-\delta) + C_\Omega\delta. $$
Finally, letting $\delta\downarrow0$, we obtain
$$ \lim_{W\to\infty} \frac{J_\Omega(W)}{W} = \frac{2^{-2\rho}\Omega}{\pi}. $$
Equivalently,
$$J_\Omega(W)= \frac{2^{-2\rho}\Omega}{\pi}W+o(W).$$
This proves the lemma. 
\end{proof}

We can now combine the preceding asymptotic estimate with the diagonal representation of the concentration operator to obtain its trace asymptotic in the Paley--Wiener setting.
\begin{lem}\label{4.11_Lemma}
Let $S=[-\Omega,\Omega]$, $I=[-R,R]$, and put
$W=R\csc\theta$. Then
\begin{align}
\operatorname{Tr}(\mathbf{L}_{I,S}^\theta)& = \frac{2^{1-2\rho}}{\pi} \Omega R \csc \theta + \operatorname{o}(R),\quad R\to \infty
\end{align}
\end{lem}

\begin{proof}
By the definition of the concentration operator,
$$ \operatorname{Tr}\big(\mathbf L_{I,S}^{\theta}\big) = \int_{-\Omega}^{\Omega} \int_{-R}^{R} \left| \mathcal G_{\alpha,\beta}^{\theta}(x,\lambda) \right|^2 \,d\sigma_{\alpha,\beta}(\lambda\csc\theta) A_{\alpha,\beta}(x)\,dx . $$
Let
$\mu=\lambda\csc\theta$, $W=R\csc\theta$. Since
$$ \mathcal G_{\alpha,\beta}^{\theta}(x,\lambda) = e^{-\frac{i}{2}(x^2+\lambda^2)\cot\theta} G_\mu^{\alpha,\beta}(x), $$
we have
$$ \operatorname{Tr}\big(\mathbf L_{I,S}^{\theta}\big) = \int_{-\Omega}^{\Omega} \int_{-R}^{R} \left|G^{\alpha,\beta}_{\mu}(x) \right|^2 \,d\sigma_{\alpha,\beta}(\lambda\csc\theta) A_{\alpha,\beta}(x)\,dx . $$
Using \eqref{abs(G_mu)2_calculation} and the odd part of the Plancherel density over $[-W,W]$, we obtain
\begin{align}\label{Trace_PW_int}
\operatorname{Tr}\big(\mathbf L_{I,S}^{\theta}\big) &= \int_{0}^{\Omega} \int_{0}^{W} \frac{|\varphi_{\mu}^{\alpha,\beta}(x)|^2}{2\pi|C_{\alpha,\beta}(\mu)|^2} \, d\mu \, A_{\alpha,\beta}(x)dx \nonumber\\
&\hspace{2cm}+ \int_{0}^{W} \frac{1}{\rho^2 + \mu^2} \left( \int_{0}^{\Omega} |(\varphi_{\mu}^{\alpha,\beta})'(x)|^2 A_{\alpha,\beta}(x)dx \right) \frac{d\mu}{2\pi|C_{\alpha,\beta}(\mu)|^2}.
\end{align}
By the definition in Lemma $\ref{jacobi_weyl}$, the first term is precisely
$J_\Omega(W)$. For the second term, using the Jacobi equation in Sturm--Liouville form \eqref{Strum_Liouville}, and integration by parts over $[0,\Omega]$ gives
\begin{align}\label{integration_by_parts_in_PW}
\int_0^\Omega |(\varphi_\mu^{\alpha,\beta})'(x)|^2 A_{\alpha,\beta}(x)\,dx =& A_{\alpha,\beta}(\Omega) \varphi_\mu^{\alpha,\beta}(\Omega) (\varphi_\mu^{\alpha,\beta})'(\Omega) \nonumber\\ 
&+ (\mu^2+\rho^2) \int_0^\Omega |\varphi_\mu^{\alpha,\beta}(x)|^2 A_{\alpha,\beta}(x)\,dx . 
\end{align}
The boundary term at $x=0$ vanishes because
$(\varphi_\mu^{\alpha,\beta})'(0)=0.$
Substituting \eqref{integration_by_parts_in_PW} into \eqref{Trace_PW_int}, we obtain
\begin{align}\label{final_trace}
\operatorname{Tr}\big(\mathbf L_{I,S}^{\theta}\big) ={}& 2J_\Omega(W)+ \int_0^W \frac{ A_{\alpha,\beta}(\Omega) \varphi_\mu^{\alpha,\beta}(\Omega) (\varphi_\mu^{\alpha,\beta})'(\Omega) }{ \rho^2+\mu^2 } \frac{d\mu} {2\pi|C_{\alpha,\beta}(\mu)|^2}. 
\end{align}

For fixed $\Omega>0$, the Jacobi--Bessel asymptotic used in Lemma
$\ref{jacobi_weyl}$ gives
\begin{align}\label{Asymptotic_PW_1}
 A_{\alpha,\beta}(\Omega)^{1/2} \varphi_\mu^{\alpha,\beta}(\Omega) = O(\mu^{-\alpha-\frac12}), \quad \mu\to\infty.
 \end{align}
For the derivative, recall the identity
\begin{align}\label{derivative_varphi}
(\varphi_\mu^{\alpha,\beta})'(x) = -\frac{\rho^2+\mu^2}{4(\alpha+1)} \sinh(2x) \varphi_\mu^{\alpha+1,\beta+1}(x).
\end{align}
Since $\Omega$ is fixed, \eqref{derivative_varphi} yields
\begin{align}\label{Asymptotic_PW_2}
A_{\alpha,\beta}(\Omega)^{1/2} (\varphi_\mu^{\alpha,\beta})'(\Omega) = O(\mu^{-\alpha+\frac12}).
\end{align}
Moreover, by Stirling's formula,
\begin{align}\label{Asymptotic_PW_3}
|C_{\alpha,\beta}(\mu)|^{-2} = O(\mu^{2\alpha+1}). 
\end{align}
Combining \eqref{Asymptotic_PW_1}, \eqref{Asymptotic_PW_2}, and \eqref{Asymptotic_PW_3}, we obtain
$$\frac{ A_{\alpha,\beta}(\Omega) \varphi_\mu^{\alpha,\beta}(\Omega) (\varphi_\mu^{\alpha,\beta})'(\Omega) }{ \rho^2+\mu^2 } \frac{1} {2\pi|C_{\alpha,\beta}(\mu)|^2}= O\left( \mu^{-\alpha-\frac12} \mu^{-\alpha+\frac12} \mu^{-2} \mu^{2\alpha+1} \right) = O(\mu^{-1}).$$
Thus
$$ \int_1^W \frac{ A_{\alpha,\beta}(\Omega) \varphi_\mu^{\alpha,\beta}(\Omega) (\varphi_\mu^{\alpha,\beta})'(\Omega) }{ \rho^2+\mu^2 } \frac{d\mu} {2\pi|C_{\alpha,\beta}(\mu)|^2} = O(\log W).$$
The contribution from $0\le\mu\le1$ is bounded. Hence \eqref{final_trace} gives
$$\operatorname{Tr}\big(\mathbf L_{I,S}^{\theta}\big) = 2J_\Omega(W)+O(\log W).$$
By Lemma $\ref{jacobi_weyl}$,
$$ J_\Omega(W) = \frac{2^{-2\rho}}{\pi}\Omega W+o(W). $$
Since $|S|=2\Omega$ and $W=R\csc\theta$,
we conclude that
\begin{align} 
\operatorname{Tr}\big(\mathbf L_{I,S}^{\theta}\big) &= \frac{2^{1-2\rho}}{\pi}\Omega W+o(W) \nonumber\\
&= \frac{2^{-2\rho}}{\pi}(2\Omega)W+o(W)=\frac{2^{-2\rho}}{\pi} |S|R\csc\theta+o(R). \nonumber 
\end{align}
\end{proof}

We can now state our final density result for $PW_{\alpha,\beta}^{\theta}(S)$.
\begin{theorem}
Let $S=[-\Omega,\Omega]$, $\alpha\ge\beta\ge-\frac12$, $\alpha>-\frac12$, and $0<\theta<\pi$. Suppose that
$$ \operatorname{Tr}\bigl(\mathbf{L}_{I,S}^{\theta}-(\mathbf{L}_{I,S}^{\theta})^2\bigr)=o(R), \quad I=[-R,R]. $$

If $\Lambda$ is a separated sampling set for $PW_{\alpha,\beta}^{\theta}(S)$, then
$$ \liminf_{R\to\infty}\frac{n([-R,R])}{R} \ge \frac{2^{1-2\rho}}{\pi}\Omega\csc\theta. $$

If $\Lambda$ is a separated interpolation set, then
$$ \limsup_{R\to\infty}\frac{n([-R,R])}{R} \le \frac{2^{1-2\rho}}{\pi}\Omega\csc\theta. $$
\end{theorem}
\begin{proof}
By the finite-window spectral duality described above, the same min--max argument used in Lemmas \ref{sampling_eigenvalue} and \ref{interpolation_eigenvalue} applies to $PW_{\alpha,\beta}^{\theta}(S)$. Let \(d>0\) be such that the points of \(\Lambda\) are separated by at least \(2d\), and set
$I_+=[-R-d,R+d]$, $I_{-}=[-R+d,R-d]$. Hence, there exists $0<\gamma<1$, independent of $R$, such that
$$ N_\delta(I)\le n(I^+),\quad \gamma<\delta<1, $$
where
$$ N_\delta(I)=\#\{k:\lambda_k(I,S)>\delta\}. $$
Since $0\le\lambda_k(I,S)\le1$,
$$ N_\delta(I) \ge \operatorname{Tr}(\mathcal L_{I,S}^{\theta}) - \frac{1}{1-\delta} \operatorname{Tr}\left( \mathcal L_{I,S}^{\theta} - (\mathcal L_{I,S}^{\theta})^2 \right). $$
By Lemma \ref{4.11_Lemma} and the assumed trace-defect estimate,
$$ N_\delta(I) \ge \frac{2^{1-2\rho}}{\pi}\Omega R\csc\theta+o(R). $$
Therefore,
$$ n(I^+) \ge \frac{2^{1-2\rho}}{\pi}\Omega R\csc\theta+o(R). $$
Since enlarging $I=[-R,R]$ by a fixed amount does not affect the limiting density, we obtain
$$ \liminf_{R\to\infty} \frac{n([-R,R])}{R} \ge \frac{2^{1-2\rho}}{\pi}\Omega\csc\theta. $$
For interpolation, the corresponding min--max argument gives, for some fixed $0<\delta<1$,
$$ n(I^-)\le N_\delta(I). $$
Moreover,
$$ N_\delta(I) \le \operatorname{Tr}(\mathcal L_{I,S}^{\theta}) + \frac1{\delta} \operatorname{Tr}\!\left( \mathcal L_{I,S}^{\theta} - (\mathcal L_{I,S}^{\theta})^2 \right). $$
Again using Lemma \ref{4.11_Lemma} and the hypothesis,
$$ n(I^-) \le \frac{2^{1-2\rho}}{\pi}\Omega R\csc\theta+o(R).$$
Shrinking $I$ by a fixed amount does not affect the limiting density, and consequently
$$ \limsup_{R\to\infty} \frac{n([-R,R])}{R} \le \frac{2^{1-2\rho}}{\pi}\Omega\csc\theta. $$
This proves both assertions.
\end{proof}

\section{Conclusion}
We developed a sampling framework for the FrOC transform. Using the Sturm--Liouville structure of Jacobi functions, we obtained a Riesz-basis representation and an explicit sampling formula whose nodes are determined by zeros of a shifted Jacobi function. We also established Landau-type necessary density conditions for sampling and interpolation through concentration-operator estimates. For $\theta=\pi/2$, the results reduce to the corresponding non-fractional Opdam--Cherednik setting. In future work, we will study sufficient density conditions and characterize complete sampling and interpolation sets for the FrOC Paley--Wiener spaces. We will also study quantitative stability of the sampling expansion under perturbations of the sampling nodes and measurement noise in the FrOC setting.

\section*{Acknowledgments}
The second author is partially supported by the XJTLU Research Development Fund (RDF-23-01-027).

\section*{Conflict of Interest}
The authors declare that they have no conflicts of interest.

\section*{Data Availability}
Data sharing is not applicable to this article as no new data were created or analyzed in this study.


\end{document}